\documentclass[11pt,oneside,a4paper,mathscr]{amsart}

\usepackage{pgfplots}
 
\pgfplotsset{compat = newest}

\usepackage{amsthm,amssymb,latexsym,amsmath}
\usepackage[all]{xy}
\usepackage{tikz}
\usepackage{pstricks-add}
\usepackage{pst-func}
\usepackage{color}
\usepackage[english]{babel}

\usepackage{amscd,euscript,mathtools}
\usepackage{xcolor}
\usepackage{graphicx}

\usepackage{comment}
\excludecomment{confidential}

\newcommand{\Z}{{\mathbb Z}}

\newcommand{\R}{\mathbb{R}}
\newcommand{\HH}{\mathbb{H}}
\newcommand{\p}[1]{{\mathbb{P}^{#1}}}

\newcommand{\op}[1]{{\mathcal O}_{\mathbb{P}^{#1}}}

\newcommand{\ox}{{\mathcal O}_{X}}

\newcommand{\OC}{{\mathcal O}_{C}} \newcommand{\IC}{I_{C}}

\newcommand{\ch}{\operatorname{ch}}
\newcommand{\td}{\operatorname{td}}

\newcommand{\calh}{{\mathcal H}}

\newcommand{\calm}{{\mathcal M}}

\newcommand{\dbx}{D^{\rm b}(X)}
\newcommand{\coh}{\operatorname{Coh}}

\newcommand{\cohab}{\mathcal{A}^{\alpha,\beta}}

\newcommand{\labs}{\lambda_{\alpha,\beta,s}}
\newcommand{\free}[1]{\mathcal{F}_{#1}}
\newcommand{\tors}[1]{\mathcal{T}_{#1}}

\newcommand{\obeta}{{\overline{\beta}}}
\newcommand{\oalpha}{\overline{\alpha}}

\newcommand{\Ext}{\operatorname{Ext}}

\newcommand{\Hom}{\operatorname{Hom}}

\DeclareMathOperator{\im}{im}

\newcommand{\into}{\hookrightarrow}
\newcommand{\onto}{\twoheadrightarrow}

\newtheorem{theorem}{Theorem}[section]

\newtheorem{proposition}[theorem]{Proposition}

\newtheorem{lemma}[theorem]{Lemma}
\newtheorem{corollary}[theorem]{Corollary}

\theoremstyle{definition}
\newtheorem{remark}[theorem]{Remark}
\newtheorem{example}[theorem]{Example}
\newtheorem{definition}[theorem]{{\bf Definition}}

\DeclarePairedDelimiter{\floor}{\lfloor}{\rfloor}

\title[Vertical asymptotics on 3-folds]{Vertical asymptotics for Bridgeland stability conditions on 3-folds}

\author{Marcos Jardim}
\address{IMECC - UNICAMP \\ Departamento de Matem\'atica \\
Rua S\'ergio  Buarque de Holanda, 651\\ 13083-970 Campinas-SP, Brazil}
\email{jardim@ime.unicamp.br}

\author{Antony Maciocia}
\address{University of Edinburgh\\School of Mathematics\\The King's Buildings\\ Peter Guthrie Tait Road\\Edinburgh\\ EH9 3FD}
\email{A.Maciocia@ed.ac.uk}

\author{Cristian Martinez}
\address{IMECC - UNICAMP \\ Departamento de Matem\'atica \\
Rua S\'ergio  Buarque de Holanda, 651\\ 13083-970 Campinas-SP, Brazil}
\email{cristian@ime.unicamp.br}

\begin{document}

\begin{abstract}
Let $X$ be a smooth projective threefold of Picard number one for which the generalized Bogomlov-Gieseker inequality holds. We characterize the limit Bridgeland semistable objects at large volume in the vertical region of the geometric stability conditions associated to $X$ in complete generality and provide examples of asymptotically semistable objects. In the case of the projective space and $\ch^\beta(E)=(-R,0,D,0)$, we prove that there are only a finite number of nested walls in the $(\alpha,s)$-plane. Moreover, when $R=0$ the only semistable objects in the outermost chamber are the 1-dimensional Gieseker semistable sheaves, and when $\beta=0$ there are no semistable objects in the innermost chamber. In both cases, the only limit semistable objects of the form $E$ or $E[1]$ (where $E$ is a sheaf) that do not get destabilized until the innermost wall are precisely the (shifts of) instanton sheaves.
\end{abstract}
\maketitle
\tableofcontents

\newpage

\section{Introduction}

The general notion of stability introduced by Bridgeland in \cite{B07} and \cite{B08} is often referred to as a generalization of Mumford stability on curves to objects in the derived category of an arbitrary smooth projective variety. Moreover, in the case of surfaces, Bridgeland's notion of stability has been proven to include  Gieseker stability. It is precisely this identification of Gieseker stability with a Bridgeland stability condition what allows for the study of geometric properties of the Gieseker moduli spaces, by analyzing the transformations that they undergo when moving along paths in the stability manifold. On a general K3 surface, for instance, a theorem of Bayer and Macrí \cite{BM14a} proves that these variations account for the whole MMP on the Gieseker moduli. 

On threefolds, however, stability conditions are only known to exist in a handful of examples. When the Picard number is one, these examples include Fano threefolds \cite{BMT,M-p3,S-q3,Li}, abelian threefolds \cite{MP,BMS}, and the quintic threefold \cite{Li2}. In each of these cases, the conjectural generalized Bogomolov-Gieseker inequality proposed by Bayer, Macrí, and Toda in \cite{BMT} holds, proving that the function  
$$
Z_{\alpha,\beta,s}(E)=-\ch_3^{\beta}(E)+\left(s+\frac{1}{6}\right)\alpha^2\ch_1^{\beta}(E)+i\left(\ch_2^{\beta}(E)-\frac{\alpha^2}{2}\ch_0(E)\right),
$$
where $(\alpha,\beta)\in\mathbb{H}=\{(x,y)\in\mathbb{R}^2\colon y>0\}$ and $s>0$, is the central charge of a stability condition $\sigma_{\alpha,\beta,s}$ on a heart $\mathcal{A}^{\alpha,\beta}$, which is obtained as a double tilt of $\coh(X)$. Let us denote by $\lambda_{\alpha,\beta,s}$ the corresponding slope function. There is no general argument proving that for some triple $(\alpha,\beta,s)$ (or any stability condition), the only $\lambda_{\alpha,\beta,s}$-semistable objects of a given Chern character are precisely the Gieseker semistable sheaves. At the moment, the closest we can get to a result of this type is through the concept of asymptotic stability introduced in \cite{JM19}.

\begin{definition}\label{asym s-st}
Let $\gamma\colon [0,\infty)\rightarrow \mathbb{H}$ be an unbounded path. An object $A\in\dbx$ is \emph{asymptotically $\labs$-(semi)stable along $\gamma$} if the following two conditions hold for a given $s>0$: 
\begin{itemize}
\item[(i)] there is $t_0>0$ such that $A\in\mathcal{A}^{\gamma(t)}$ for every $t>t_0$;
\item[(ii)] for every sub-object $F\into A$ within $\mathcal{A}^{\gamma(t)}$ with $t>t_0$, there is $t_1>t_0$ such that $\lambda_{\gamma(t),s}(F) < ~(\le)~ \lambda_{\gamma(t),s}(A)$ for $t>t_1$. 
\end{itemize}
\end{definition}
This notion of stability is the first step to identify a possible Gieseker chamber in the stability manifold. Indeed, as proven in \cite{JM19} and \cite{P}, there are paths (with $\alpha$ fixed and $\beta\rightarrow -\infty$), for which asymptotic stability coincides with Gieseker stability.  However, given a Chern character $v$, a bound for the Bridgeland walls for $v$ intersecting the path $\gamma$ and depending only $v$ is still unknown.

In this article, we will concentrate in the study of asymptotic stability along the paths of fixed $\beta=\overline{\beta}$. Contrary to the case of the paths considered in \cite{JM19} and \cite{P}, these paths provide a set up closer to the set up for surfaces in the following sense:
\begin{itemize}
    \item[(i)] the heart $\coh^{\obeta}(X)$ is fixed,
    \item[(ii)] there is a 1-parameter family $\mathcal{A}^{\alpha,\obeta}$ of tilts of $\coh^{\obeta}(X)$,
    \item[(iii)] for each heart $\mathcal{A}^{\alpha,\obeta}$, there is a 1-parameter family of slopes $\lambda_{\alpha,\obeta,s}$.
\end{itemize}
This picture is particularly close to the surface case when we study  1-dimensional sheaves. Indeed, we will prove that given a Chern character $v=(0,0,D,S)$, we can choose a path $\{\beta=\obeta\}$ intersecting only a finite number of walls such that in the unbounded chamber the only semistable objects of Chern character $v$ are precisely the 1-dimensional Gieseker semistable sheaves. Moreover, for this value $\beta=\obeta$, letting the parameter $s$ vary we obtain finitely many nested walls in the $(\alpha,s)$-slice.

We should remark that unlike in the case of sheaves whose support has dimension $\geq 2$, where some kind of stability can be studied by using tilt-stability, 1-dimensional sheaves need to be studied by using $\lambda_{\alpha,\beta,s}$-stability since they are always tilt-semistable by definition.

In the positive rank case, asymptotic stability along the paths $\{\beta=\obeta\}$ is far from being Gieseker stability as the asymptotically semistable objects are often two-term complexes. However, examples of asymptotically semistable objects include shifts of reflexive sheaves and, when $X=\mathbb{P}^3$, instanton sheaves. In fact, in this last example the picture is similar to the one of 1-dimensional sheaves.

The organization of the paper is as follows. In Section \ref{Bg} we fix the notation and recall some known facts about stability conditions on threefolds. In Section \ref{labs sst 0} we state and prove our main general theorems (Theorem \ref{vert-prop-1} and  Theorem \ref{vert-prop-2}), characterizing asymptotically semistable objects along $\{\beta=\obeta\}$; in this section we also provide the first examples of asymptotically semistable objects, and completely characterize asymptotically semistable objects of positive rank and $\obeta$ irrational (Theorem \ref{vert-thm-1}), and when $v_1=0$ for $\obeta$ in a small interval depending on $v$ and $s$ (Theorem \ref{vert-thm-2} and Theorem \ref{vert-thm-3}). In Section \ref{sec:instantons}, for $X=\mathbb{P}^3$ we prove that the shift of an instanton sheaf is asymptotically $\lambda_{\alpha,\obeta,s}$-semistable for $\obeta\leq 0$ and $s>1/3$; we also completely describe the asymptotically  $\lambda_{\alpha,0,s}$-semistable objects of Chern characters $(-2,0,1,0)$ and $(-2,0,2,0)$. In Section \ref{sec:bridgeland_walls} we use the generalized Bogomolov-Gieseker inequality to prove that there are finitely many walls in the $(\alpha,s)$-slice for the twisted Chern character $\ch^{\obeta}=(-R,0,D,0)$; in the case $R=0$ we prove that we recover Gieseker semistability in the outermost chamber. In Section \ref{sec:beta=0} we specialize to the case $X=\mathbb{P}^3$ and $\obeta=0$; here we prove that the innermost chamber is contained in a quiver region, that is bounded by a wall destabilizing all the shifts of instanton sheaves, and that there are no $\lambda_{\alpha,0,s}$-semistable objects of Chern character $(-R,0,D,0)$ in the interior of the innermost chamber; we also compute all the walls and destabilizing objects in the $(\alpha,s)$-slice for the Chern characters $(-2,0,1,0)$, $(-2,0,2,0)$, and $(0,0,3,0)$ to demonstrate our algorithm. 

\subsection*{Acknowledgments}
MJ is supported by the CNPQ grant number 302889/2018-3 and the FAPESP Thematic Project 2018/21391-1. CM is supported by the FAPESP grant number 2020/06938-4, which is part of the FAPESP Thematic Project 2018/21391-1. 


\section{Background material}\label{Bg}

\subsection{Notation} In what follows, $X$ will denote a smooth projective threefold of Picard rank 1 with ample generator $H$ for which stability conditions are known to exist, i.e., $X$ is either Fano \cite{M-p3,S-q3,Li}, or abelian \cite{MP,BMS}, or a quintic threefold \cite{Li2}. Other than specified we will use the following notation:
\begin{itemize}
\item For $\beta\in \mathbb{R}$ and $A\in D^b(X)$, the twisted Chern character of $A$ is
\begin{align*}
\ch^{\beta}(A):=& e^{-\beta H}\cdot\ch(A)\\
=&\ch^{\beta}_0(A)+\ch^{\beta}_1(A)H+\ch^{\beta}_2(A)H^2+\ch^{\beta}_3(A)H^3\\
=:&(\ch^{\beta}_0(A),\ch^{\beta}_1(A),\ch^{\beta}_2(A),\ch^{\beta}_3(A)).
\end{align*}
\item $\delta_{ij}^{\beta}(F,A) := \ch_i^{\beta}(F)\ch_j^{\beta}(A)-\ch_j^{\beta}(F)\ch_i^{\beta}(A)$.
\item $\rho_{A}(\alpha,\beta) = \ch_2^{\beta}(A) -\alpha^2\ch_0^{\beta}(A)/2$.
\item $\ch^{\alpha,\beta}(A):=\operatorname{Re}\bigl(\exp(-\beta H-i \alpha H)\cdot \ch(A)\bigr).$
\item Let $\mathcal{A}$ be the heart of a t-structure on $D^b(X)$. For an object $A\in D^b(X)$, we denote by $\mathcal{H}^{i}_{\mathcal{A}}(A)$
the cohomology objects of $A$ with respect to the t-structure generated by $\mathcal{A}$. We suppress $\mathcal{A}$ from the notation in the case of the standard t-structure. 

\item We will write  $A_i:=\calh^{-i}_{\mathcal{A}}(A)$ and $A_{ij}:=\calh^{-j}(A_i)$. In the case when $A_i=0$ for $i\neq 0,1$ and $A_{ij}=0$ for $j\neq 0,1$, we have the three distinguished triangles:
\begin{equation} \label{3ts}
\begin{gathered}
A_{11}[1]\to A_1\to A_{10} ; \\
A_{01}[1]\to A_0\to A_{00} ; \\
A_1[1]\to A\to A_0.
\end{gathered}
\end{equation}
\end{itemize}

\subsection{Stability conditions on threefolds}
We briefly recall the construction of Bridgeland stability conditions on $X$ due to Bayer, Macr\'i, and Toda \cite{BMT}.

For $\beta\in\mathbb{R}$ define the (twisted) Mumford slope by
$$
\mu_{\beta}(E)=\begin{cases}\frac{\ch^{\beta}_1(E)}{\ch^{\beta}_0(E)}& \text{if}\ \ch^{\beta}_0(E)\neq 0,\\ +\infty &\text{if}\ \ch^{\beta}_0(E)=0.\end{cases}
$$ 
The subcategories 
\begin{align*}
\tors{\beta}&:=\{E\in \coh(X)\ |\ \mu_{\beta}(Q)>0\ \text{for every quotient}\ E\onto Q\},\ \text{and}\\
\free{\beta}&:=\{E\in \coh(X)\ |\ \mu_{\beta}(F)\leq 0\ \text{for every subsheaf} \ F\into E\},
\end{align*}
form a torsion pair on $\coh(X)$. The corresponding tilted category is the extension closure
$$
\coh^{\beta}(X):=\langle \free{\beta}[1],\ \tors{\beta}\rangle.
$$
For every $\alpha>0$ we can define the following ``slope'' function on $\coh^{\beta}(X)$
$$
\nu_{\alpha, \beta}(E)=\begin{cases}\dfrac{\ch^{\beta}_2(E)-\frac{\alpha^2}{2}\ch^{\beta}_0(E)}{\alpha\ch^{\beta}_1(E)}& \text{if}\ \ch^{\beta}_1(E)\neq 0,\\ +\infty &\text{if}\ \ch^{\beta}_1(E)=0.\end{cases}
$$ 
We refer to $\nu_{\alpha,\beta}$ as the tilt, and to objects in $\coh^{\beta}(X)$ that are semistable with respect to $\nu_{\alpha,\beta}$ as tilt--semistable objects. Analogously to the case of Mumford slope, the subcategories
\begin{align*}
\mathcal{T}_{\alpha,\beta}&:=\{E\in \coh^{\beta}(X)\ |\ \nu_{\alpha,\beta}(Q)>0\ \text{for every quotient}\ E\onto Q\ \text{in}\ \coh^{\beta}(X)\},\ \text{and}\\
\mathcal{F}_{\alpha,\beta}&:=\{E\in \coh^{\beta}(X)\ |\ \nu_{\alpha,\beta}(F)\leq 0\ \text{for every subobject} \ F\into E\ \text{in}\ \coh^{\beta}(X)\},
\end{align*}
form a torsion pair on $\coh^{\beta}(X)$. The corresponding tilted category is the extension closure
$$
\mathcal{A}^{\alpha,\beta}:=\langle {\mathcal{F}}_{\alpha,\beta}[1],\ \tors{\alpha,\beta}\rangle.
$$
As proven in \cite{BMT,M-p3,BMS,Li}, we know that $\mathcal{A}^{\alpha,\beta}$ supports a 1-parameter family of stability conditions whose central charges are given by
$$
Z_{\alpha,\beta,s}(A)=-\left(\ch^{\beta}_3(A)-\left(s+\frac{1}{6}\right)\alpha^2\ch^{\beta}_1(A)\right)+i\left(\ch^{\beta}_2(A)-\frac{\alpha^2}{2}\ch^{\beta}_0(A)\right),\ s>0.
$$ 
The associated Bridgeland slope is
$$
\lambda_{\alpha,\beta,s}(A):=\frac{\ch^{\beta}_3(A)-\left(s+\frac{1}{6}\right)\alpha^2\ch^{\beta}_1(A)}{\ch^{\beta}_2(A)-\frac{\alpha^2}{2}\ch^{\beta}_0(A)}.
$$
The stability conditions $\sigma_{\alpha,\beta,s}=(Z_{\alpha,\beta,s},\mathcal{A}^{\alpha,\beta})$ are geometric (i.e., skyscraper sheaves $\mathcal{O}_x$ are stable of the same phase), and they satisfy the support property with respect to the quadratic form
$$
Q_{\alpha,\beta, K}(\ch):=(\ch^{\beta})^{t}B_{\alpha,K}\ch^{\beta},\ \text{where}\ B_{\alpha,K}=\begin{pmatrix}0 & 0 & -K\alpha^2& 0\\ 0 & K\alpha^2 &0 &-3\\ -K\alpha^2 & 0 & 4 &0 \\ 0 & -3 & 0 & 0\end{pmatrix},
$$
for every $K\in[1, 6s+1)$ (see \cite{BMT,BMS}). For an object $A\in\mathcal{A}^{\alpha,\beta}$ we write $Q_{\alpha,\beta, K}(A)$ for $Q_{\alpha,\beta, K}(\ch(A))$, and denote the corresponding pairing 
$$
(\ch^{\beta}(F))^{t}B_{\alpha,K}\ch^{\beta}(A)
$$ 
by $Q_{\alpha,\beta, K}(F,A)$.

The following lemma is a consequence of the support property and will play an important role in the section that follows.
\begin{lemma}[{\cite[Lemma 2.7]{S}}]\label{qform}
\begin{enumerate}
\item Let $A,B\in\mathcal{A}^{\alpha,\beta}$ be $\sigma_{\alpha,\beta,s}$--semistable objects. If $\lambda_{\alpha,\beta,s}(A)=\lambda_{\alpha,\beta,s}(B)$, then $Q_{\alpha,\beta,K}(A,B)\geq 0$.
\item Assume that there is a wall defined by an exact sequence 
$$
0\rightarrow A \rightarrow E \rightarrow B \rightarrow 0
$$ 
in $\mathcal{A}^{\alpha,\beta}$. Then $0\leq Q_{\alpha,\beta,K}(A)+Q_{\alpha,\beta,K}(B)\leq Q_{\alpha,\beta,K}(E)$. 
\end{enumerate}
\end{lemma}






\section{Asymptotic $\labs$-stability along vertical lines} \label{labs sst 0}

In this section, we will focus on the characterization of asymptotically $\labs$-semistable objects along vertical lines. In contrast with the case of horizontal lines studied in \cite{JM19}, the set of asymptotically $\labs$-semistable objects along a vertical line $\{\beta=\obeta\}$ with fixed Chern character $v$ does depend both on $\obeta$ and on the parameter $s$.

More precisely, we establish the following two results. The first one provides sufficient conditions for asymptotic $\labs$-stability along vertical lines.

\begin{theorem} \label{vert-prop-1}
If an object $A\in\dbx$ with $\ch_0(A)\ne0$ satisfies the following conditions
\begin{enumerate}
\item $\calh^p(A)=0$ for $p\ne-1,0$;
\item $\dim\calh^0(A)\le1$, and every quotient sheaf $\calh^0(A)\onto P$ (including $\calh^0(A)$ itself) satisfies
$$ \dfrac{\ch_3(P)}{\ch_2(P)} > \dfrac{6s+1}{3}\left( \mu(A)-\overline{\beta} \right) + \overline{\beta} $$
whenever $\ch_2(P)\ne0$;
\item $\tilde{A}=\calh^{-1}(A)$ is $\mu$-stable;
\item $\tilde{A}$ is either reflexive or $\tilde{A}^{**}/\tilde{A}$ has pure dimension 1, and every subsheaf $R\into\tilde{A}^{**}/\tilde{A}$ (including $\tilde{A}^{**}/\tilde{A}$ itself) satisfies
$$ \dfrac{\ch_3(R)}{\ch_2(R)} < \dfrac{6s+1}{3}\left( \mu(A)-\obeta \right) + \obeta ; $$ 
\item if $U$ is a sheaf of dimension at most 1 and $u:U\to \calh^0(A)$ is a non-zero morphism that lifts to a monomorphism $\tilde{u}:U\into A$ within $\mathcal{A}^{\alpha,\obeta}$ for every $\alpha\gg0$, then
$$ \dfrac{\ch_3(U)}{\ch_2(U)} < \dfrac{6s+1}{3}\left( \mu(A)-\obeta \right) + \obeta; $$
\end{enumerate}
then $A$ is asymptotically $\labs$-stable along the vertical line $\{\beta=\obeta\}$ for every $s\geq1/3$.
\end{theorem}

Note that conditions (2) and (3) above imply that $\ch(A)$ satisfies the Bogomolov inequality. The restriction on $s$ is due to the possible presence of isolated "bubble" walls which can appear as $s$ decreases. Potentially, there could be infinitely many of these for a fixed value of $s$ and that would prevent us being able to deduce $\lambda_{\alpha,\beta,s}$-asymptotic stability. There is no evidence that this does occur and we would conjecture that there are only ever finitely many walls (above a fixed value of $\alpha$) for any $s>0$. 

A partial converse to the previous statement, describing necessary conditions for asymptotic $\labs$-semistability, is the following.

\begin{theorem} \label{vert-prop-2}
Let $v$ be a numerical Chern character satisfying $v_0\ne0$ and the Bogomolov inequality $v_1^2-2v_0v_2\geq 0$. If $A\in\dbx$ is an asymptotically $\labs$-semistable object along the vertical line $\{\beta=\obeta\}$ with $\ch(A)=v$, then:
\begin{enumerate}
\item $\calh^{-2}(A)=0$; 
\item $\dim\calh^{0}(A)\le1$, and every sheaf quotient $\calh^{0}(A)\onto P$ (including $\calh^{0}(A)$ itself) satisfies
$$ \dfrac{\ch_3(P)}{\ch_2(P)} \ge \dfrac{6s+1}{3}\left( \mu(A)-\overline{\beta} \right) + \overline{\beta} $$
whenever $\ch_2(P)\ne0$;
\item $\calh^{-1}(A)=\tilde{A}$ is $\mu$-semistable, and every sub-object $F\into A$ with $\mu(F)=\mu(A)$ satisfies
$$\left(s+\dfrac{1}{6}\right)(\mu(A)-\obeta)\delta_{20}^{\obeta}(F,A)-\dfrac{1}{2}\delta_{30}^{\obeta}(F,A) \le 0 ;$$
\item $\tilde{A}$ is either reflexive or $\tilde{A}^{**}/\tilde{A}$ has pure dimension 1, and every subsheaf $R\into\tilde{A}^{**}/\tilde{A}$ (including $\tilde{A}^{**}/\tilde{A}$ itself) satisfies
$$ \dfrac{\ch_3(R)}{\ch_2(R)} \le \dfrac{6s+1}{3}\left( \mu(A)-\obeta \right) + \obeta ; $$ 
\item if $U$ is a sheaf of dimension at most 1 and $u:U\to A_{00}$ is a non-zero morphism that lifts to a monomorphism $\tilde{u}:U\into A$ within $\mathcal{A}^{\alpha,\obeta}$ for every $\alpha\gg0$, then
$$ \dfrac{\ch_3(U)}{\ch_2(U)} \le \dfrac{6s+1}{3}\left( \mu(A)-\obeta \right) + \obeta . $$
\end{enumerate}
\end{theorem}

Before we step into the proofs, let us present two examples of objects satisfying the conditions of Theorem \ref{vert-prop-1}.

\begin{example}\label{refl mu-st}
Let $S$ be a $\mu$-stable reflexive sheaf on $X$. One easily checks that $A:=S[1]$ satisfies the five conditions of Theorem \ref{vert-prop-1}, so $A$ is asymptotically $\labs$-stable along any vertical line, for every $s\ge1/3$.

Furthermore, assume that $S$ is not locally free, and let $Z$ be a 0-dimensional sheaf whose support lies in its singular set; it follows that
$$ \Ext^1(Z,S[1]) \simeq \Ext^2(Z,S) \simeq \Ext^1(S,Z)\ne0 , $$
hence there is a non-trivial extension $A\in\dbx$ given by the triangle 
$$ S[1] \to A \to Z. $$
We now argue that such $A$ is asymptotically $\labs$-stable along any vertical line $\{\beta=\obeta\}$ as long as no subsheaf of $Z$ factors through $A$.

Indeed, conditions (1) through (4) of Theorem \ref{vert-prop-1} are immediately seen to hold for any $\obeta$. To verify the fifth condition, assume that the sheaf morphism $u:U\to\mathcal{O}_{p}$ lifts to a monomorphism $u:U\into A$ in $\cohab$; set $G:=A/U$. Letting $K:=\ker u$, note that the short exact sequence $0\to S \to \tilde{G}\to K\to0$ in $\coh(X)$ forces $K=0$, since $S$ is reflexive and $\tilde{G}=\calh^{-1}(G)$ is torsion free (cf. Lemma \ref{vert-lem-3}). It follows that $u$ must be an injection, which would contradict our assumption. 

An example of such $Z$ is $Z=\mathcal{O}_p$, where $p$ is a singular point of $S$. In fact, by the same argument above, a lifting to $A$ of a sheaf morphism $U\rightarrow \mathcal{O}_p$ will necessarily split $A$ as a direct sum $S[1]\oplus\mathcal{O}_{p}$.\hfill\qed
\end{example}

\begin{example}\label{stability ic}
If $C$ be a curve in $X$ (that is, be a subscheme of pure dimension 1), then, by Theorem \ref{vert-prop-1}, $A:=\IC[1]$ is asymptotically $\labs$-stable along $\{\beta=\obeta\}$ provided every subsheaf $R\into\OC$ satisfies
\begin{equation}\label{cond-c}
\dfrac{\ch_3(R)}{\ch_2(R)} < -\dfrac{6s-2}{3}\obeta.
\end{equation}
Conversely, if $A:=\IC[1]$ is asymptotically $\labs$-stable, then Theorem \ref{vert-prop-2} implies that the inequality in display (\ref{cond-c}) is satisfied. 

Note that if $C$ is irreducible and reduced and $R\into\OC$ is a subsheaf, then 
$$ \ch_2(R)=\ch_2(\OC)=\deg(C) ~~{\rm and}~~ $$
$$ \ch_3(R) \le \ch_3(\OC) = \chi(\OC)-\td_1(X)\deg(C) \le 1-\ch_1(TX)/2. $$
Therefore, if $\ch_1(TX)\ge2$, then the inequality in equation (\ref{cond-c}) is automatically satisfied whenever $s>1/3$ and $\obeta<0$. If $\obeta=0$ and $\ch_1(TX)>2$ then the inequality in equation \eqref{cond-c} is satisfied for all $s>0$.\hfill\qed
\end{example}

\begin{remark}
Notice that unlike in the case of surfaces, the limit semistable objects along a vertical line are far from being (shifts of) Gieseker semistable sheaves. For instance, if $Z\subset X$ is a 0-dimensional subscheme then $I_Z[1]\in\mathcal{A}^{\alpha,\obeta}$ for all $\obeta$ and $\alpha>>0$; however, $I_Z[1]$ does not satisfy Theorem \ref{vert-prop-2}(4) and so it is not limit semistable. 
\end{remark}

As illustrated in Example \ref{stability ic}, the vertical line $\{\beta=\mu(v)\}$ also plays an important role, and we must consider vertical lines lying to its left and to its right in separate cases. First, let us start with a general results for fixed $\obeta\in\R$.

\begin{lemma}\label{vert-lem-1}
If there is $\alpha_0>0$ such that $A\in\mathcal{A}^{\alpha,\overline{\beta}}$ for all $\alpha>\alpha_0$, then $\ch_0(A_0)\leq0$, $0\leq \ch_0(A_1)\leq -\ch_0(A)$, $A_{11}=0$, and $\dim A_{00}\le2$.
\end{lemma}

\begin{proof}
Since $A_1[1]$ and $A_0$ are both objects in $\mathcal{A}^{\alpha,\overline{\beta}}$ whenever $A$ is, \cite[Proposition 2.3]{JM19} implies that
$$ \rho_{A_0}(\alpha,\overline{\beta})\geq0 ~~{\rm and}~~ \rho_{A_{1}}(\alpha,\overline{\beta})\leq0. $$
However, note that
$$ \lim_{\alpha\to\infty} \dfrac{1}{\alpha^2} \rho_{A_k}(\alpha,\overline{\beta}) =
-\dfrac{1}{2} \ch_0(A_k) $$
for $k=0,1$. It follows that $\ch_0(A_0)\leq0$ and $\ch_0(A_1)\geq0$. Since $\ch(A)=\ch(A_0)-\ch(A_1)$, we also conclude that $\ch_0(A_1)\leq -\ch_0(A)$.

By Lemma \cite[Lemma 2.8]{JM19}, $A_{11}$ is a reflexive sheaf; since $A_{11}[2]\in\mathcal{A}^{\alpha,\overline{\beta}}$, \cite[Proposition 2.3]{JM19} implies that
$\rho_{A_{11}}(\alpha,\overline{\beta})\geq0$. By the same argument as above, we conclude that $\ch_0(A_{11})\leq0$, which forces $A_{11}$ to be the zero sheaf.

If $\ch^{\overline{\beta}}_1(A_{00})=0$, then $\rho_{A_{00}}(\alpha,\overline{\beta})\ge0$ for all $\alpha>0$ by \cite[Proposition 2.1]{JM19}, thus $\ch_0(A_{00})=\ch_1(A_{00})=0$.

It follows that $\ch^{\overline{\beta}}_1(A_{00})>0$ since $A_{00}\in\tors\obeta$. Since $A_{00}$ is a quotient of $A_0\in\mathcal{T}_{\alpha,\overline{\beta}}$, we have  
$$ \nu_{\alpha,\overline{\beta}}(A_{00}) > 0 ~\Longrightarrow~
\lim_{\alpha\to\infty} \dfrac{1}{\alpha^2}\nu_{\alpha,\overline{\beta}}(A_{00}) = \dfrac{-\ch_0(A_{00})}{2\ch^{\overline{\beta}}_1(A_{00})}\ge0 , $$
hence $\ch_0(A_{00})=0$.

\end{proof}


\subsection{Proof of Theorem \ref{vert-prop-1}}

The first step for proving Theorem \ref{vert-prop-1} is to check that an object $A\in\dbx$ satisfying the hypotheses does belong to $\mathcal{A}^{\alpha,\obeta}$ when $\alpha$ is sufficiently large. We fix $\obeta\in\R$ and start by considering the case $\mu(A)>\obeta$.

\begin{lemma}\label{mu-sst}
If $E$ is a $\mu_{\le2}$-semistable sheaf with $\mu(E)>\obeta$, then there exists $\alpha_0>0$, such that $E\in\free{\alpha,\obeta}$ for every $\alpha>\alpha_0$.
\end{lemma}
\begin{proof}
Since $E$ is $\mu$-semistable, we have that $E\in\tors{\obeta}\subset\coh^{\obeta}(X)$. Moreover, \cite[Theorem 5.2]{JM19} implies that $E$ is $\nu_{\alpha,\obeta}$-semistable whenever $\alpha>R$, where $R$ denotes the radius of the unique actual $\nu$-wall for $E$ (this exists by \cite[Lemma 5.3(2)]{JM19}).  If $F\into E$ is a sub-object within $\coh^{\obeta}(X)$, then there is an $\alpha_0>0$, given by the larger between $R$ and any positive solution of the equation $\rho_E(\alpha,\obeta)=0$, such that for every $\alpha>\alpha_0$
$$ \nu_{\alpha,\obeta}(F) \le \nu_{\alpha,\obeta}(E) \le 0, $$
It follows that $E\in\mathcal{F}_{\alpha,\obeta}$.
\end{proof}

\begin{lemma}\label{vert-lem-2}
If $A\in\dbx$ is an object with $\mu(A)>\obeta$ satisfying the following conditions
\begin{enumerate}
\item $\calh^p(A)=0$ for $p\ne-1,0$;
\item $\dim\calh^0(A)\le1$;
\item $\calh^{-1}(A)$ is a $\mu_{\le2}$-semistable sheaf;
\end{enumerate}
there is $\alpha_0>0$ such that $A\in\mathcal{A}^{\alpha,\overline{\beta}}$ for all $\alpha>\alpha_0$ and $\obeta<\mu(A)$.
\end{lemma}
\begin{proof}
Let $E:=\calh^{-1}(A)$ and $T:=\calh^0(A)$, so $A$ fits into the following triangle
$$ E[1] \to A \to T \to E[2]. $$
Since $\mu(E)=\mu(A)$, Lemma \ref{mu-sst} implies that there is $\alpha_0>0$ such that $E\in\free{\alpha,\obeta}$ whenever $\alpha>\alpha_0$. Since $\dim A_{00}\le1$, we have that $A_{00}$ belongs to both $\tors{\beta}$ and $\tors{\alpha,\beta}$ for every $(\alpha,\beta)\in\HH$. It follows that $A\in\mathcal{A}^{\alpha,\obeta}$ whenever $\alpha>\alpha_0$, as desired.


\end{proof}

The next step is to obtain an analogous statement for $\obeta\ge\mu(A)$, which requires a slightly different approach.

\begin{lemma} \label{vert-lem-4}
Let $E$ be a $\mu$-semistable sheaf with $\mu(E)<\obeta$ and such that every quotient sheaf $E\onto G$ with $\mu(G)=\mu(E)$ and $0<\ch_0(G)<\ch_0(E)$ satisfies $\delta_{20}(E,G)\ge0$. Then there exists $\alpha_0>0$ such that $E[1]\in\tors{\alpha,\obeta}$ whenever $\alpha>\alpha_0$.
\end{lemma}
\begin{proof}
First, note that $E\in\free{\obeta}$ when $\obeta\ge\mu(E)$, so $E[1]\in\coh^{\obeta}(X)$. Let $E[1]\onto G$ be a quotient object within $\coh^{\obeta}(X)$; clearly, $\calh^0(G)=0$, so $G=K[1]$ for some sheaf $K\in\free{\obeta}$. Letting $F$ be the kernel of the epimorphism  $E[1]\onto G$, we obtain the following short exact sequence in $\coh(X)$:
\begin{equation}
0 \to \calh^{-1}(F) \to E \to K \stackrel{f}{\rightarrow} \calh^{0}(F) \to 0.
\end{equation}
Assume that $\calh^{0}(F)\ne0$ and let $K'=:\ker f$; note that $\mu(K)\le \obeta < \mu(\calh^{0}(F))$ since $K\in\free{\obeta}$ while $\calh^{0}(F)\in\tors{\obeta}$, thus $\mu(K')<\mu(K)=\mu(G)$. it follows that $\mu(E)\le\mu(K')<\mu(G)$, where the first inequality comes from the $\mu$-semistability of $E$, since $K'$ is a quotient sheaf of $E$. When $\calh^{0}(F)=0$, then $K$ is a quotient sheaf of $E$, thus $\mu(E)\le\mu(G)$. 

Therefore we have
$$ \lim_{\alpha\to\infty} \dfrac{1}{\alpha^2} \left( \nu_{\alpha,\obeta}(E[1]) - \nu_{\alpha,\obeta}(G) \right) = 
\dfrac{\delta_{10}(E,G)}{2\ch_1^{\obeta}(K)\ch_1^{\obeta}(E)} \le 0, $$
since both terms in the denominator are negative. If $\mu(E)=\mu(G)$, then $\calh^{0}(F)$ must vanish, so $0<\ch_0(G)<\ch_0(E)$ and we have
$$ \nu_{\alpha,\obeta}(E[1]) - \nu_{\alpha,\obeta}(G) = 
\dfrac{\mu(E)-\obeta}{\ch_1^{\obeta}(K)\ch_1^{\obeta}(E)} \delta_{20}(E,K) \le 0. $$

We conclude that for each quotient object $E[1]\onto G$ within $\coh^{\obeta}(X)$ there is $\alpha_0>0$ such that $\nu_{\alpha,\obeta}(G) > \nu_{\alpha,\obeta}(E[1]) > 0$ whenever $\alpha>\alpha_0$. In other words, $E[1]\in\tors{\alpha,\obeta}$ for $\alpha>\alpha_0$, as desired. 
\end{proof}

\begin{example}
Here is an example of a strictly $\mu$-semistable sheaf on $X=\p3$ that satisfies the hypothesis of Lemma \ref{vert-lem-4}. Let $S$ be a strictly $\mu$-semistable rank 2 reflexive sheaf on $\p3$ with $\ch_1(S)=0$; choose an epimorphism $e:S\onto\mathcal{O}_p$ where $p$ is a point in $\p3$, and take $E:=\ker e$. One can check that if $E\onto G$ is a torsion free quotient sheaf with $\mu(G)=0$, then $G=I_{C}$ for a curve $C\subset\p3$ of degree $\deg C=-\ch_2(S)=-\ch_2(E)$. It follows that
$$ \delta_{20}(E,I_{C})=\ch_2(E)+2\deg(C)=\deg(C)>0, $$
as desired. Notice that such sheaves are not Gieseker semistable.
\hfill\qed
\end{example}

\begin{lemma}\label{vert-lem-5}
If $E$ is a $\mu$-semistable sheaf, then $E[1]\in\tors{\alpha,\mu}\subset\mathcal{A}^{\alpha,\mu}$ for every $\alpha>0$ where $\mu=\mu(E)$.
\end{lemma}
\begin{proof}
We already noted, in the proof of Lemma \ref{vert-lem-4}, that $E\in\free{\mu}$, so $E[1]\in\coh^\mu(X)$. If $E[1]\onto G$ is a quotient within $\coh^\mu(X)$, then $\ch_1^\mu(G)=0$, thus $\nu_{\alpha,\mu}(G)=+\infty$ for every $\alpha>0$. It follows that $E[1]\in\tors{\alpha,\mu}$, as desired.
\end{proof}

With the previous two results in hand, we can now state a version of Lemma \ref{vert-lem-2} for $\obeta\ge\mu(A)$.

\begin{lemma}\label{vert-lem-6}
If $A\in\dbx$ is an object with $\mu(A)\le\obeta$ satisfying the following conditions
\begin{enumerate}
\item $\calh^p(A)=0$ for $p\ne-1,0$;
\item $\dim\calh^0(A)\le1$;
\item $\calh^{-1}(A)$ is a $\mu$-semistable sheaf such that every quotient sheaf $E\onto G$ with $\mu(G)=\mu(E)$ and $0<\ch_0(G)<\ch_0(E)$ satisfies $\delta_{20}(E,G)\ge0$;
\end{enumerate}
then there is $\alpha_0>0$ such that $A\in\mathcal{A}^{\alpha,\overline{\beta}}\cap\coh^{\obeta}(X)$ for all $\alpha>\alpha_0$.
\end{lemma}
\begin{proof}
Again, set $E:=\calh^{-1}(A)$ and $T:=\calh^0(A)$, and fix $\obeta\ge\mu(A)=\mu(E)$; the first hypothesis yields a triangle $E[1]\to A\to T$. Lemmas \ref{vert-lem-4}  and \ref{vert-lem-5} imply that $E\in\free{\obeta}$ and $E[1]\in\tors{\alpha,\obeta}$ for $\alpha\gg0$. In addition, it is easy to see that $T$ also belongs to $\tors{\obeta}$ and $\tors{\alpha,\obeta}$ for every $\alpha$. It follows that $A$ belongs to $\mathcal{A}^{\alpha,\obeta}\cap\coh^{\obeta}(X)$ for $\alpha\gg0$.
\end{proof}

To prove Theorem \ref{vert-prop-1} we fix an object $A$ satisfying the conditions of Theorem \ref{vert-prop-1}. Our strategy will be to take a subobject $F\into A$ which destabilizes ``to infinity" and show that this limit contradicts the conditions on $A$. But to do this we need to be able to show such an object exists. We consider actual $\lambda$-walls for $A$ along a vertical line $\{\beta=\obeta\}$, by which we mean points $(\alpha,\obeta)$ where $A$ is strictly $\lambda_{\alpha,\obeta,s}$-semistable. We let $\alpha_0>0$ be given by Lemmas \ref{vert-lem-2} and \ref{vert-lem-6}, so that $A\in\mathcal{A}^{\alpha,\overline{\beta}}$ for $\alpha>\alpha_0$. Note that unbounded walls can only cross $\beta=\overline{\beta}$ at most once because the equation for the wall is linear in $\alpha^2$. But bounded walls can and do cross vertical walls twice.

\begin{lemma}\label{vert_finite}
Fix $\obeta\in\mathbb{R}$ and $s\geq1/3$.  Then there are only finitely many actual $\lambda$-walls for $A$  intersecting the vertical line $\beta=\overline{\beta}$ for $\alpha>\alpha_0$. 
\end{lemma}
\begin{proof}
There are only finitely many actual $\lambda$-walls intersecting $\Theta_{v}$, so we will assume the walls do not intersect $\Theta_{v}$. If a connected component of a wall is a bubble then at $s=1/3$ it has a minimum which must lie on $\Xi_{u,v}$. As it is a closed curve this must intersect $\Xi_{u,v}$ again which must also be on $\Theta_v$. So we can assume the walls are not bubbles. But now any bounded actual wall intersecting $\{\beta=\obeta\}$ for $\alpha\gg0$ must traverse the region between the branches of $\Theta_v$ and, above $\alpha_0$ this is compact and so there can only be finitely many of these by local finiteness.

Arguing in a similar way, we see that the only possibility for infinitely many actual $\lambda$-walls crossing $\{\beta=\obeta\}$ is if there are infinitely many unbounded $\lambda$-walls. The component crossing $\{\beta=\obeta\}$ must have two unbounded branches (on either side of $\{\beta=\obeta\}$). But this cannot happen for $s\geq1/3$ by the form for the homogenization $\overline{f}$ of $f$ at infinity given in \cite[Equation 28]{JM19} as only $\{\beta=\beta'\}$ is an asymptote, where $\beta'$ is given in the statement of \cite[Lemma 4.8]{JM19}.
\end{proof}

\begin{lemma} \label{vert-lem-3}
If $A\in\dbx$ and $\overline{\beta}$ satisfy the conditions of Lemma \ref{vert-lem-2}, and \begin{equation} \label{star}
0 \to F \to A \to G \to 0
\end{equation}
is a short exact sequence in $\mathcal{A}^{\alpha,\overline{\beta}}$ for every $\alpha>\alpha_0$, then:
\begin{enumerate}
\item $\calh^{-2}(F)=0$ and $\calh^{-2}(G)=\calh^{-1}(\calh^{0}_{\obeta}(G))=0$;
\item $\dim \calh^{0}(F)\le2$, and $\dim \calh^{0}(G)\le1$;
\item $\calh^{-1}(F)$ and $\calh^{-1}(G)$ are torsion free.
\end{enumerate}
\end{lemma}

In particular, it follows that if $\ch_0(F)=\ch_0(A)$, then $\ch_0(\tilde{G})=\ch_0(G)=0$ thus $\tilde{G}=0$.

\begin{proof}
Since $F,G\in\mathcal{A}^{\alpha,\overline{\beta}}$ for $\alpha\gg0$, Lemma \ref{vert-lem-1} implies (in the notation of \cite[Section 2.3]{JM19}) that $F_{11}=G_{11}=0$ and $\dim F_{00}\le2$; $G_{00}$ is a quotient of $A_{00}$, so $\dim G_{00}\le1$ by hypothesis.

We denote $E:=\calh^{-1}(A)=\tilde{A}$ and $T:=\calh^{0}(A)$. Applying $\calh^{*}$ to the sequence in display (\ref{star}), we obtain:
\begin{equation} \label{star2}
0 \to \tilde{F} \to E \to \tilde{G} \to F_{00} \to T \to G_{00} \to 0
\end{equation}
$\tilde{F}$ is clearly torsion free, since it is a subsheaf of $E$. 

Since $G_{01}\in\free{\overline{\beta}}$, we have that $\mu^+(G_{01})\le\overline{\beta}$. On the other hand, let $G'_{01}$ be the image of the composite morphism $E\to \tilde{G}\onto G_{01}$. Since $E$ is $\mu$-semistable, it follows that $\mu(G_{01}')\ge\mu(E)=\mu(A)$, hence $\mu^+(G_{01})\ge\mu(A)$. However, $\overline{\beta}<\mu(A)$ by assumption, so $G_{01}=0$. We also conclude that $\tilde{G}\simeq G_{1}\in\free{\alpha,\obeta}$.

Now let $K_1:=E/\tilde{F}$, and let $T_K$ be its torsion subsheaf; using the sequence $0\to\tilde{F}\to E\to K_1\to 0$ and the corresponding morphisms between the double dual sheaves, we conclude that $T_K$ is a subsheaf of $\tilde{F}^{**}/\tilde{F}$, thus $\dim T_K\le1$. In addition, $K_1$ is a subsheaf of $\tilde{G}$, so let $K_2:=\tilde{G}/K_1$, which is also in $\coh_{\le1}$ because it is a subsheaf of $F_{00}$. Therefore, the torsion subsheaf of $\tilde{G}$ would also have dimension at most one, contradicting the fact that $\tilde{G}\in\free{\alpha,\obeta}$. Thus both $K$ and $\tilde{G}$ must be torsion free sheaves.
\end{proof}

We are finally ready to complete the proof of Theorem \ref{vert-prop-1}.

\noindent\emph{Proof of Theorem \ref{vert-prop-1} for $\obeta<\mu(A)$.}
Set $E:=\calh^{-1}(A)$ and $T:=\calh^0(A)$. Lemma \ref{vert-lem-2} implies that $A$ satisfies condition (i) of Definition \ref{asym s-st}.

Let $F\into A$ be a sub-object within $\mathcal{A}^{\alpha,\obeta}$ for all $\alpha$ sufficiently large, and let $G:=A/F$ (quotient within $\mathcal{A}^{\alpha,\obeta}$). Lemma \ref{vert-lem-3} implies that $F_{11}=G_{11}=G_{01}=0$, $\dim F_{00}\le 2$  and $\dim G_{00}\le1$, so $\ch_i(G)=-\ch_i(\tilde{G})$ for $i=0,1$ and $\ch_0(F)=-\ch_0(\tilde{F})$ so that 
$$ \mu(F) = \mu(\tilde{F}) - \dfrac{\ch_1(F_{00})}{\ch_0(\tilde{F})} \le \mu(\tilde{F}). $$
Moreover, the cohomology sheaves of the objects $F$, $A$ and $G$ satisfy the sequence (\ref{star2}), thus, in particular, $0\le\ch_0(\tilde{F})\le\ch_0(E)$.

We start by considering the case $\ch_0(F)=\ch_0(\tilde{F})=0$, which is equivalent to assuming that $\tilde{F}=0$ since $\tilde{F}$ is torsion free. Indeed, this assumption implies that $F\simeq F_{00}$ and reduces the sequence (\ref{star2}) to 
\begin{equation}\label{star3}
0\to E \to \tilde{G} \to F_{00} \stackrel{u}{\to} T \to G_{00} \to 0. 
\end{equation}
If $u=0$, then $T \simeq G_{00}$ and $E^{**}\simeq\tilde{G}^{**}$, so that $F_{00}$ becomes a subsheaf of $E^{**}/E$, so $F_{00}$ must have pure dimension 1. The inequality in item (4) implies that
$$ \lim_{\alpha\to\infty} \left( \lambda_{\alpha,\overline{\beta},s}(F)-\lambda_{\alpha,\overline{\beta},s}(A) \right) =  \dfrac{\ch_3(F_{00})}{\ch_2(F_{00})} - \obeta - \dfrac{6s+1}{3}\left( \mu(A)-\obeta \right) < 0, $$
i.e. there exists $\alpha_1>0$ such that $\lambda_{\alpha,\overline{\beta},s}(F)<\lambda_{\alpha,\overline{\beta},s}(A)$ for every $\alpha>\alpha_1$, as desired.

Therefore, we can assume that $u:F_{00}\to T$ is a non-zero morphism that lifts to a monomorphism $F_{00}\to A$; the same conclusion will be reached by considering the inequality in item (5).

Assume now that $\ch_0(A)<\ch_0(F)<0$, or, equivalently, $0<\ch_0(\tilde{F})<\ch_0(E)$; since $E$ is $\mu$-stable, we conclude that
$$ \mu(F) \le \mu(\tilde{F})<\mu(E)=\mu(A). $$
It then follows that 
$$ \lim_{\alpha\to\infty} \left( \lambda_{\alpha,\overline{\beta},s}(F)-\lambda_{\alpha,\overline{\beta},s}(A) \right) = \dfrac{6s+1}{3}\left( \mu(F)-\mu(A) \right) < 0, $$

Finally, we consider the case $\ch_0(F)=\ch_0(A)$, so that $\ch_0(\tilde{G})=\ch_0(G)=0$, thus $\tilde{G}=0$ since it is torsion free by Lemma \ref{vert-lem-3}. It follows that $G\simeq G_{00}$ and, assuming that $\ch_2(G)\ne0$, we have
$$ \lim_{\alpha\to\infty} \left( \lambda_{\alpha,\overline{\beta},s}(A)-\lambda_{\alpha,\overline{\beta},s}(G) \right) =
\dfrac{6s+1}{3}\left( \mu(A)-\obeta \right) - \left( \dfrac{\ch_3(G_{00})}{\ch_2(G_{00})} -\obeta \right) < 0. $$
The last inequality follows from condition (2), since $G_{00}$ is a quotient of $A_{00}$.
\hfill\qed

\medskip

\noindent\emph{Proof of Theorem \ref{vert-prop-1} for $\obeta\ge\mu(A)$.}
Let $F\into A$ be a sub-object within $\mathcal{A}^{\alpha,\obeta}$, with quotient $G:=A/F$, for $\alpha>\alpha_0$. Lemma \ref{vert-lem-1} imply that $G_{11}=0$ and $\dim F_{00}\le 2$; we also have that $\dim G_{00}\le 1$ by hypothesis, since $G_{00}$ is a quotient of $A_{00}$. Moreover, $\ch_0(F)=-\ch_0(\tilde{F})$, and $\mu(F) \le \mu(\tilde{F})$.


Since $\ch_0(A)\le\ch_0(F)\le0$, it is enough to consider three cases: $\ch_0(F)=0$, $\ch_0(A)<\ch_0(F)<0$ and $\ch_0(F)=\ch_0(A)$, just as in the proof when $\obeta<\mu(A)$. The main difference is that we cannot conclude, like in Lemma \ref{vert-lem-3}, that $\tilde{G}$ is torsion free; however, any torsion subsheaf $P\into \tilde{G}$ must actually be a subsheaf of $G_{10}=G_1\in\free{\alpha,\obeta}$, since $G_{01}$ is torsion free, thus $\dim P=2$.

If $\ch_0(F)=0$, then $\tilde{F}=0$, so $F=F_{00}$. If $\ch_1(F)=\ch_1(F_{00})>0$, then
$$ \lim_{\alpha\to\infty} \dfrac{1}{\alpha^2}\left( \lambda_{\alpha,\overline{\beta},s}(A)-\lambda_{\alpha,\overline{\beta},s}(F) \right) =
\dfrac{6s+1}{6}\dfrac{\ch_1(F)}{\ch_2(F)-\ch_1(F)\obeta} > 0, $$
because the denominator on the right hand side of the equality coincides with $\rho_F(\oalpha,\beta)$ and therefore must be positive. if follows that $\lambda_{\alpha,\overline{\beta},s}(F)<\lambda_{\alpha,\overline{\beta},s}(A)$ for $\alpha\gg0$, as desired.

If $\ch_1(F)=0$, then, by examining the induced exact sequence in $\coh$ like the one in display \eqref{star3}, implying that the torsion subsheaf of $\tilde{G}$ would be a subsheaf of $F_{00}$, which is impossible because the latter has dimension at most 1. Thus we conclude that $\tilde{G}$ must be torsion free in this case, and the proof proceeds as before.

The case $\ch_0(A)<\ch_0(F)<0$ goes exactly as in the proof when $\obeta<\mu(A)$, but the third possibility requires further details.

The equality $\ch_0(F)=\ch_0(A)$ implies that
$$ \ch_0(G)=-\ch_0(\tilde{G})=-\ch_0(G_{10})-\ch_0(G_{01})=0, $$
thus $G_{01}=0$ since it is a torsion free sheaf, and $\tilde{G}=G_{10}=G_1$, so $\tilde{G}[1]\in\mathcal{A}^{\alpha,\obeta}$. 

If $\ch_1(G)=-\ch_1(\tilde{G})<0$, then 
$$ \lim_{\alpha\to\infty} \dfrac{1}{\alpha^2}\left( \lambda_{\alpha,\overline{\beta},s}(A)-\lambda_{\alpha,\overline{\beta},s}(G) \right) =
\dfrac{6s+1}{6}\dfrac{\ch_1(G)}{\ch_2(G)-\ch_1(G)\obeta} < 0, $$
since the denominator coincides with $\rho_G(\alpha,\obeta)$ and it is positive by \cite[Proposition 2.3]{JM19}. We conclude that $\lambda_{\alpha,\overline{\beta},s}(A) < \lambda_{\alpha,\overline{\beta},s}(G)$ hence also $\lambda_{\alpha,\overline{\beta},s}(F) < \lambda_{\alpha,\overline{\beta},s}(A)$ for $\alpha\gg0$.

If $\ch_0(G)=\ch_1(G)=0$, then $\ch_0(\tilde{G})=\ch_1(\tilde{G})=0$ because $\dim G_{00}\le1$. If $\tilde{G}\ne0$, \cite[Proposition 2.3]{JM19} implies that $\rho_{\tilde{G}[1]}=-\ch_2(\tilde{G})\ge0$, thus $\ch_2(\tilde{G})=0$ too; in addition, it also follows that $\tau_{\tilde{G}[1]}=-\ch_3(\tilde{G})>0$, providing a contradiction. 

We conclude that $G=G_{00}$, so assuming that $\ch_2(G)\ne0$, we have
$$ \lim_{\alpha\to\infty} \left( \lambda_{\alpha,\overline{\beta},s}(A)-\lambda_{\alpha,\overline{\beta},s}(G) \right)  =
\dfrac{6s+1}{3}\left( \mu(A)-\obeta \right) - \left( \dfrac{\ch_3(G_{00})}{\ch_2(G_{00})} -\obeta \right)<0, $$
where the last inequality follows from condition (2), since $G_{00}$ is a quotient of $A_{00}$.

Finally, we assume that $\ch_i(G)=0$ for $i=0,1,2$, so $\ch_3(G)=\ch_3(G_{00})>0$. Since $\ch_i(F)=\ch_i(A)$ for $i=0,1,2$, we have that
$$ \lim_{\alpha\to\infty} \alpha^2  \left( \lambda_{\alpha,\overline{\beta},s}(F)-\lambda_{\alpha,\overline{\beta},s}(A) \right) = 
\dfrac{2(\ch_3(A)-\ch_3(F))}{\ch_0(A)} = \dfrac{2\ch_3(G)}{\ch_0(A)} <0 $$
since $\ch_0(A)<0$. \hfill\qed


\subsection{Proof of Theorem \ref{vert-prop-2}}

Let us know look into the converse direction, describing the properties of asymptotically $\labs$-semistable objects along vertical lines. Applying Lemma \ref{vert-lem-1}, item (i) in Definition \ref{asym s-st} immediately implies that $A_{11}=0$ and $\dim A_{00}\le 2$. To see that, in fact, $\dim A_{00}\le 1$, just note that $A_{00}$ is a quotient of $A$ in $\mathcal{A}^{\alpha,\obeta}$ whenever $A\in\mathcal{A}^{\alpha,\obeta}$; if $\ch_1(A_{00})>0$, then
$$ \lim_{\alpha\to\infty} \dfrac{1}{\alpha^2}\left( \lambda_{\alpha,\overline{\beta},s}(A)-\lambda_{\alpha,\overline{\beta},s}(A_{00}) \right) =
\dfrac{6s+1}{6}\dfrac{\ch_1(A_{00})}{\ch_2(A_{00})-\ch_1(A_{00})\obeta} > 0, $$
contradicting the asymptotic $\labs$-semistability of $A$.

Any quotient sheaf $P$ of $A_{00}$ also belongs to $\tors{\alpha,\obeta}$ for any $\alpha$ and $\obeta$ and therefore is also a quotient of $A$ within $\mathcal{A}^{\alpha,\beta}$; thus if $\ch_2(P)\ne 0$, we obtain:
$$ \lim_{\alpha\to\infty} \left( \lambda_{\alpha,\overline{\beta},s}(P)-\lambda_{\alpha,\overline{\beta},s}(A) \right) =
\dfrac{\ch_3(P)}{\ch_2(P)} -\obeta -\dfrac{6s+1}{3}\left( \mu(A)-\obeta \right) \ge 0, $$
which yields the formula in item (2) of the proposition.

Next, note that $\ch_i(\tilde{A})=-\ch_i(A)$ for $i=0,1$; in particular, $\mu(\tilde{A})=\mu(A)$. If $\tilde{A}$ is not $\mu$-semistable, let $S$ be its maximal $\mu$-stable subsheaf. Using either Lemma \ref{vert-lem-2} or Lemma \ref{vert-lem-4}, we have that $S[1]\in\mathcal{A}^{\alpha,\obeta}$ when $\alpha$ is sufficiently large, thus the monomorphism $S\into\tilde{A}$ in $\coh(X)$ lifts to a monomorphism $S[1]\into A$ in $\mathcal{A}^{\alpha,\obeta}$. However, 
$$ \lim_{\alpha\to\infty} \left( \lambda_{\alpha,\overline{\beta},s}(S[1])-\lambda_{\alpha,\overline{\beta},s}(A) \right) = 
\dfrac{6s+1}{3}\left( \mu(S)-\mu(A) \right) > 0, $$
contradicting the asymptotic $\lambda$-semistability of $A$ along $\{\beta=\obeta\}$.

Given any sub-object $F\into A$ within $\mathcal{A}^{\alpha,\overline{\beta}}$ for $\alpha\gg0$ with $\mu(F)=\mu(A)$, note that
$$ \lim_{\alpha\to\infty} \alpha^2\left( \lambda_{\alpha,\overline{\beta},s}(F)-\lambda_{\alpha,\overline{\beta},s}(A) \right) = $$
$$ = \dfrac{4}{\ch_0(F)\ch_0(A)} \left(\left(s+\dfrac{1}{6}\right)(\mu(A)-\obeta)\delta_{20}^{\obeta}(F,A)-\dfrac{1}{2}\delta_{30}^{\obeta}(F,A)\right) \le 0; $$
since $\ch_0(F),\ch_0(A)<0$, we obtain the inequality in the third item.

Setting $Q_A:=\tilde{A}^{**}/\tilde{A}$, note that the standard short exact sequence
$$ 0 \to \tilde{A} \to \tilde{A}^{**} \to Q_A \to 0 $$
in $\coh(X)$ induces the following short exact sequence in $\mathcal{A}^{\alpha,\obeta}$, since both $\tilde{A}$ and $\tilde{A}^{**}$ are $\mu$-semistable sheaves with slope larger than $\obeta$:
$$ 0 \to Q_A \to \tilde{A}[1] \to \tilde{A}^{**}[1] \to 0 . $$
Therefore, any subsheaf $R\into Q_A$ (including $Q_A$ itself) can be regarded as a sub-object of $A$ within $\mathcal{A}^{\alpha,\obeta}$. Any 0-dimensional $R\into Q_A$ would destabilize $A$, thus $Q_A$ must have pure dimension 1. In addition,
$$ \lim_{\alpha\to\infty} \left( \lambda_{\alpha,\overline{\beta},s}(R)-\lambda_{\alpha,\overline{\beta},s}(A) \right) =
\dfrac{\ch_3(R)}{\ch_2(R)} -\obeta -\dfrac{6s+1}{3}\left( \mu(A)-\obeta \right) \le 0, $$
which yields the formula in item (4) of Theorem \ref{vert-prop-2}.

Finally, let $u:U\to A_{00}$ be a non-zero morphism which lifts to a monomorphism $\tilde{u}:U\to A$ in $\mathcal{A}^{\alpha,\obeta}$ for every $\alpha\gg0$; we can assume that $u_2>0$, for otherwise it would asymptotically $\lambda$-destabilize $A$. It follows that
\begin{equation} \label{U-->A}
\lim_{\alpha\to\infty} \left( \lambda_{\alpha,\overline{\beta},s}(U)-
\lambda_{\alpha,\overline{\beta},s}(A) \right) =
\dfrac{\ch_3(U)}{\ch_2(U)} -\obeta -\dfrac{6s+1}{3}\left( \mu(A)-\obeta \right) \le 0,
\end{equation}
as desired, completing the proof of Theorem \ref{vert-prop-2}. \hfill\qed

\begin{remark}
If $A\in\dbx$ is asymptotically $\labs$-semistable along $\{\beta=\obeta\}$, then $\calh^{-1}(\calh^0_{\obeta}(A))=0$ when $\obeta<\mu(A)$, while $\calh^{0}(\calh^{-1}_{\obeta}(A))=0$ when $\obeta>\mu(A)$.

Indeed, assume first $\obeta<\mu(A)$. Since $A_0$ is a quotient of $A$ within $\mathcal{A}^{\alpha,\overline{\beta}}$, we obtain, if $\ch_0(A_0)=\ch_0(A_{01})\ne0$ (the first equality comes from he fact that $\dim A_{00}\le1$) 
$$ \lim_{\alpha\to\infty} \left( \lambda_{\alpha,\overline{\beta},s}(A_0)-\lambda_{\alpha,\overline{\beta},s}(A) \right) =
\dfrac{6s+1}{3}\left( \mu(A_{0})-\mu(A) \right) \ge 0, $$
thus $\mu(A_{01})=\mu(A_{0})\ge\mu(A)>\obeta$, hence $\ch_1^{\overline{\beta}}(A_{01})>0$. However, $\ch_1^{\overline{\beta}}(A_0)=-\ch_1^{\overline{\beta}}(A_{01})\ge0$ by \cite[Proposition 2.1]{JM19}, since $A_0\in\coh^{\obeta}(X)$; it follows that $\ch_0(A_{01})=0$, hence $A_{01}=0$ since this is a torsion free sheaf.

Next, assume that $\obeta>\mu(A)$. Since $A_{11}=0$, then $A_1=A_{10}$, hence $A_{10}[1]$ is a sub-object of $A$ within $\mathcal{A}^{\alpha,\overline{\beta}}$. If $\ch_0(A_{10})\ne0$, then 
$$ \lim_{\alpha\to\infty} \left( \lambda_{\alpha,\obeta,s}(A_{10})-\lambda_{\alpha,\obeta,s}(A) \right) =
\dfrac{6s+1}{3}\left( \mu(A_{10})-\mu(A) \right) \le 0, $$
since $A$ is $\lambda_{\alpha,\obeta,s}$-semistable for every $\alpha>\alpha_0$, thus $\mu(A_{10})\le\mu(A)<\obeta$. It follows that $\ch_1^{\obeta}(A_{10})<0$, contradicting the fact that $A_{10}\in\mathcal{T}_{\obeta}\subset\coh^{\obeta}(X)$.

If $\ch_0(A_{10})=0$, then $\ch_1(A_{10})=\ch_1^{\obeta}(A_{10})\ge0$; assume that $\ch_1(A_{10})\ne0$. In addition, $\ch_2^{\obeta}(A_{10}) = -\rho_{A_{10}[1]}(\alpha,\obeta)\le0$, since $A_{10}[1]\in\mathcal{A}^{\alpha,\obeta}$. It then follows that
$$ \lim_{\alpha\to\infty} \dfrac{1}{\alpha^2} \left( \lambda_{\alpha,\obeta,s}(A_{10})-\lambda_{\alpha,\obeta,s}(A) \right) = -\dfrac{6s+1}{6}  \dfrac{\ch_1(A_{10})}{\ch_2^{\obeta}(A_{10})} > 0, $$
again contradicting the asymptotic $\labs$-stability of $A$.

We therefore conclude that $\ch_0(A_{10})=\ch_1(A_{10})=0$, thus $\nu_{\alpha,\obeta}(A_{10})=+\infty$. However, $A_{10}\in\mathcal{F}_{\alpha,\obeta}$, hence it follows that $A_{10}=0$; furthermore, $A=A_0\in\mathcal{T}_{\alpha,\obeta}$ for $\alpha\gg0$. \hfill\qed
\end{remark}


\subsection{Precise characterization of asymptotically (semi)stable objects}

The difficulty in providing a precise characterization of asymptotically $\labs$-stable objects along vertical lines is the existence of actual $\lambda$-walls that are asymptotic to a vertical line, as described in \cite[Lemma 4.8]{JM19}; indeed, note that the expressions on the left hand side of the inequalities in Theorems \ref{vert-prop-1} and \ref{vert-prop-2} precisely correspond to the values of $\obeta$ defined in \cite[Lemma 4.8]{JM19}. Below we describe two situations in which a precise characterization can be obtained.

\begin{theorem}\label{vert-thm-1}
Let $v$ be a numerical Chern character satisfying $v_0\ne0$ and the Bogomolov inequality $v_1^2-2v_0v_2\geq 0$; assume that $\gcd(v_0,v_1)=1$, and 
$\obeta\notin\mathbb{Q}$. An object $A\in\dbx$ with $\ch(A)=v$ is asymptotically $\labs$-stable along the vertical line $\{\beta=\obeta\}$ if and only if the conditions of Theorem \ref{vert-prop-1} are satisfied.
\end{theorem}

Alternatively, the conclusion of Theorem \ref{vert-thm-1} also holds if one assumes $s\notin\mathbb{Q}$ when $\obeta\in\mathbb{Q}$.

\begin{proof}
The \emph{if} part is just Theorem \ref{vert-prop-1}.

To prove the converse statement, just note that the five conditions in Theorem \ref{vert-prop-2} coincide with the ones in Theorem \ref{vert-prop-1} under the stated hypothesis. Indeed, the equality in items (2), (4) and (5) cannot occur when $\obeta\notin\mathbb{Q}$; as for item (3), recall that $\mu$-stability and $\mu$-semistability are equivalent when $\gcd(v_0,v_1)=1$.
\end{proof}

If $\mu(v)\in\Z$, then, after twisting by a suitable line bundle, we can assume that $\mu(v)=0$, and the Bogomolov inequality reduces to $v_2>0$ when $v_0<0$. We conclude this section by studying this case in detail.

\begin{lemma} \label{ineq's}
Let $v$ be a numerical Chern character satisfying $v_0<0$, $v_1=0$ and $v_2>0$. Given any $(\alpha,\beta)\in\HH$, let $A\in\cohab$ be an object with $\ch(A)=v$. If $\calh^{-2}(A)=0$, $\calh^{-1}(A)$ is $\mu$-semistable, and $\calh^{0}(A)\in\coh(X)_{1}$, then
\begin{enumerate}
\item $\ch_2(\calh^{0}(A))\le\ch_2(A)$;
\item $\ch_2(\calh^{-1}(A)^{**}/\calh^{-1}(A))\le\ch_2(A)$;
\item if there is a monomorphism $U\into A$ in $\cohab$ for some $U\in\coh(X)_{1}$, then $u_2\le 2\ch_2(A)$.
\end{enumerate}
\end{lemma}

\begin{proof}
For the first inequality, note that $\ch_2(A_{00})=\ch_2(A)+\ch_2(\tilde{A})\le\ch_2(A)$ because $\tilde{A}$ is $\mu$-semistable (Bogomolov inequality).

Next, note that
$$ \ch_2(\tilde{A}^{**}/\tilde{A}) = \ch_2(\tilde{A}^{**}) - \ch_2(\tilde{A}) \le - \ch_2(\tilde{A}) = \ch_2(A) - \ch_2(A_{00}) \le \ch_2(A), $$
where the first inequality follows from the $\mu$-semistability of $\tilde{A}^{**}$.

Finally, letting $G:=A/U$ and taking cohomology in $\coh(X)$, we obtain the following exact sequence:
$$ 0 \to \tilde{A} \to \tilde{G} \to U \stackrel{u}{\to} A_{00} \to G_{00} \to 0. $$
Note that $\ker u$ is a subsheaf of $\tilde{A}^{**}/\tilde{A}$, while $\im u$ is a subsheaf of $A_{00}$. Since $u_2=\ch_2(\ker u) + \ch_2(\im u)$, the third inequality follows easily from the first two.
\end{proof}

\begin{theorem}\label{vert-thm-2}
Let $v$ be a numerical Chern character satisfying $v_0<0$, $v_1=0$ and $v_2>0$. For each $s<1/3$ there is $\epsilon>0$ (depending on $s$ and $v$) such that an object $A\in\dbx$ with $\ch(A)=v$ is asymptotically $\labs$-semistable along the vertical line $\{\beta=\obeta\}$ for $-\epsilon<\obeta<0$ if and only if the following conditions hold:
\begin{enumerate}
\item $\calh^p(A)=0$ for $p\ne-1,0$;
\item $\dim\calh^{0}(A)\le1$, and every sheaf quotient $\calh^{0}(A)\onto P$ (including $\calh^{0}(A)$ itself) satisfies $\ch_3(P)\ge0$ whenever $\ch_2(P)\ge0$;
\item $\tilde{A}=\calh^{-1}(A)$ is $\mu$-semistable, and every sub-object $F\into A$ in $\mathcal{A}^{\alpha,\obeta}$ for $\alpha\gg0$ and with $\mu(F)=0$ and $\ch_0(A)<\ch_0(F)<0$ satisfies $\delta_{30}(F,A)\geq0$, and if $\delta_{30}(F,A)=0$, then $\delta_{20}(F,A)\geq0$
\item $\tilde{A}^{**}/\tilde{A}$ has pure dimension 1, and every subsheaf $R\into\tilde{A}^{**}/\tilde{A}$ (including $\tilde{A}^{**}/\tilde{A}$ itself) satisfies $\ch_3(R)<0$;
\item if $U$ is a sheaf of dimension at most 1 and $u:U\to\calh^{0}(A)$ is a non-zero morphism that lifts to a monomorphism $\tilde{u}:U\into A$ in $\mathcal{A}^{\alpha,\obeta}$ for every $\alpha\gg0$, then $u_2>0$ and $u_3<0$.
\end{enumerate}
\end{theorem}

\begin{proof}
First, assume that $A\in\dbx$ is asymptotically $\labs$-semistable along the $\{\beta=\obeta\}$, so that item (1) follows immediately; the strategy is to reduce the three inequalities in Theorem \ref{vert-prop-2} to the ones above when $\obeta$ is taken suitably small.

The inequality in item (4) of Theorem \ref{vert-prop-2} reduces to
$$ \dfrac{\ch_3(R)}{\ch_2(R)} \le \dfrac{2-6s}{3}\obeta < 0, $$
since $s<1/3$ and $\obeta<0$. The same argument for the inequality in item (5) of Theorem \ref{vert-prop-2}, and it reduces to $u_3<0$.

Next, take 
$$ \dfrac{-3}{(2-6s)\ch_2(A)} < \obeta < 0. $$
If $A_{00}\onto P$ is a quotient sheaf with $\ch_2(P)>0$, then
$$ \dfrac{\ch_3(R)}{\ch_2(R)} \ge \dfrac{2-6s}{3}\obeta > \dfrac{-1}{\ch_2(A)} \ge \dfrac{-1}{\ch_2(R)}, $$
by the inequality in item (2) of Lemma \ref{ineq's}; Therefore, thus $\ch_3(P)>-1$, which is equivalent to $\ch_3(P)\ge0$.

The fact that $\tilde{A}$ is $\mu$-semistable is also guaranteed by Theorem \ref{vert-prop-2}. If $F\into A$ is a sub-object within $\mathcal{A}^{\alpha,\obeta}$ for $\alpha\gg0$ and $\mu(F)=0$, then the inequality in item (3) of Theorem \ref{vert-prop-2} reduces to
$$ \dfrac{3s-1}{3}\delta_{20}(F,A)\obeta + \dfrac{1}{2}\delta_{30}(F,A) \ge 0. $$
Note that if $\delta_{30}(F,A)=0$, then the previous inequality implies that $\delta_{20}(F,A)\ge0$.

Note that $\delta_{20}(F,A)$ can only assume finitely many values. Indeed, boundedness guarantees that there are only finitely many possibilities for $\ch(\tilde{F})$ among all subsheaves $\tilde{F}\into\tilde{A}$ with $\mu(\tilde{F})=\mu(F)=0$; in addition, an argument similar to the one used to prove item (3) of Lemma \ref{ineq's} can be used to show that $0\le\ch_2(F_{00})\le2\ch_2(A)$. This means that $\ch_2(F)$, hence $\delta_{20}(F,A)$, can only assume finitely many values, as desired. 

If $\delta_{20}(F,A)\le0$, then
$$ \delta_{30}(F,A) \ge \dfrac{2-6s}{3}\delta_{20}(F,A)\obeta \ge 0 . $$
Set
$$ \delta_{20}^{+}(F,A):=\min\{\delta_{20}(F,A) ~~|~~ F\into A ~~{\rm and}~~ \delta_{20}(F,A)>0 \}; $$
taking 
$$ \dfrac{3}{(6s-2)\delta_{20}^{+}}\le\obeta<0 $$
we conclude that 
$$ \delta_{30}(F,A) \ge \dfrac{2-6s}{3}\delta_{20}(F,A)\obeta \ge 
\dfrac{2-6s}{3}\delta_{20}^{+}(F,A)\obeta > -1 $$
thus again $\delta_{30}(F,A)\ge0$.

For the converse statement, the strategy is to compare with the inequalities provided in Theorem \ref{vert-prop-1}. First note that
$$ \ch_3(P)\ge0 ~\Rightarrow~ \dfrac{\ch_3(P)}{\ch_2(P)}\ge0>\dfrac{2-6s}{3}\obeta, $$ 
since $\obeta<0$ and $s<1/3$, providing the inequality in item (2) of Theorem \ref{vert-prop-1}.

Let $R\into\tilde{A}^{**}/\tilde{A}$ be a subsheaf; taking
$$ -\dfrac{3}{(2-6s)\ch_2(A)} < \obeta < 0 $$
provides us with the following implications:
$$ \ch_3(R)<0 ~\Leftrightarrow~ \ch_3(R)\le-1 ~\Leftrightarrow~ \dfrac{\ch_3(R)}{\ch_2(R)} \le \dfrac{-1}{\ch_2(R)} \le \dfrac{-1}{\ch_2(A)} < \dfrac{2-6s}{3}\obeta, $$
since $\ch_2(R)\le\ch_2(A)$ by the inequality in item (2) of Lemma \ref{ineq's}. This yields the inequality in item (4) of Theorem \ref{vert-prop-1}; the inequality in item (5) of Theorem \ref{vert-prop-1} is obtained in the same way, using the inequality in item (3) of Lemma \ref{ineq's}.

Following the proof of Theorem \ref{vert-prop-1}, it only remains for us to study sub-objects $F\into A$ within $\mathcal{A}^{\alpha,\obeta}$ when $\alpha>\alpha_0$ for some $\alpha_0>0$ and $\obeta$ is sufficiently close to 0 and $\ch_0(A)<\ch_0(F)<0$. Note that
$$ \lim_{\alpha\to\infty} \left( \lambda_{\alpha,\overline{\beta},s}(F)-\lambda_{\alpha,\overline{\beta},s}(A) \right) = \dfrac{6s+1}{3}\mu(\tilde{F}) \le 0, $$
since $\tilde{F}\into\tilde{A}$ is a subsheaf and $\tilde{A}$ is $\mu$-semistable. If $\mu(F)=\mu(\tilde{F})=0$, then
$$ \lim_{\alpha\to\infty} \alpha^2 \left( \lambda_{\alpha,\overline{\beta},s}(F)-\lambda_{\alpha,\overline{\beta},s}(A) \right) = $$
$$ \dfrac{-4}{\ch_0(F)\ch_0(A)} \left( \dfrac{3s-1}{3}\delta_{20}(F,A)\obeta + \dfrac{1}{2}\delta_{30}(F,A) \right).  $$
If $\delta_{30}(F,A)>0$, we can take $\obeta$ sufficiently close to 0 to make the sum inside the parenthesis positive, thus
$$ \lim_{\alpha\to\infty} \alpha^2 \left( \lambda_{\alpha,\overline{\beta},s}(F)-\lambda_{\alpha,\overline{\beta},s}(A) \right) < 0. $$
If $\delta_{30}(F,A)=0$, then $\delta_{20}(F,A)\ge0$ again makes the term inside the parenthesis positive when $\delta_{20}(F,A)>0$. Finally, if $\delta_{20}(F,A)=\delta_{30}(F,A)=0$, then $\lambda_{\alpha,\overline{\beta},s}(F)=\lambda_{\alpha,\obeta,s}(A)$ for every $\alpha\gg0$.
\end{proof}

The same argument can be used to prove a similar statement for the case $s>1/3$.

\begin{theorem}\label{vert-thm-3}
Let $v$ be a numerical Chern character satisfying $v_0<0$, $v_1=0$ and $v_2>0$. For each $s>1/3$ there is $\epsilon>0$ (depending on $s$ and $v$) such that an object $A\in\dbx$ with $\ch(A)=v$ is asymptotically $\labs$-semistable along the vertical line $\{\beta=\obeta\}$ for $-\epsilon<\obeta<0$ if and only if the following conditions hold:
\begin{enumerate}
\item $A_{11}=A_{01}=0$;
\item $\dim A_{00}\le1$, and every sheaf quotient $A_{00}\onto P$ (including $A_{00}$ itself) satisfies $\ch_3(P)\ge0$ whenever $\ch_2(P)>0$;
\item $\tilde{A}$ is $\mu$-semistable, and every sub-object $F\into A$ in $\mathcal{A}^{\alpha,\obeta}$ for $\alpha\gg0$ and with $\mu(F)=0$ and $\ch_0(A)<\ch_0(F)<0$ satisfies $\delta_{30}(F,A)\geq0$, and if $\delta_{30}(F,A)=0$, then $\delta_{20}(F,A)\leq0$
\item $\tilde{A}^{**}/\tilde{A}$ has pure dimension 1, and every subsheaf $R\into\tilde{A}^{**}/\tilde{A}$ (including $\tilde{A}^{**}/\tilde{A}$ itself) satisfies $\ch_3(R)\leq0$;
\item if $U$ is a sheaf of dimension at most 1 and $u:U\to A_{00}$ is a non-zero morphism that lifts to a monomorphism $\tilde{u}:U\into A$ in $\mathcal{A}^{\alpha,\obeta}$ for every $\alpha\gg0$, then $u_2>0$ and $u_3\leq0$.
\end{enumerate}
\end{theorem}

\begin{example}\label{example dual instanton}
Let $L$ be a line in $X=\p3$, with $\mathcal{O}_{L}$ being its structure sheaf; note that $\Ext^1(\mathcal{O}_{L}(1),\op3[1])\simeq H^0(\mathcal{O}_{L}(1))\ne0$. Therefore, a pair of nontrivial, linearly independent sections of $H^0(\mathcal{O}_{L}(1))$ yields an object $A\in D^{\rm b}(\p3)$ given by the triangle
$$ 2\cdot\op3[1] \to A \to \mathcal{O}_{L}(1), $$
which does not split as a direct sum; note that $\ch(A)=(-2,0,1,0)$. We will now argue that such $A\in D^{\rm b}(\p3)$ is asymptotically $\labs$-stable along $\{\beta=\obeta\}$ when $0\le\obeta<3/(6s-2)$ for $s>1/3$. Note that we cannot simply apply Theorem \ref{vert-prop-1}, since $\tilde{A}=2\cdot\op3$ is not $\mu$-stable.

First note that $2\cdot\op3\in\cohab$ whenever $\alpha>|\beta|$ while $\mathcal{O}_{L}(1)\in\cohab$ for every $(\alpha,\beta)$. It follows that $A\in\cohab$ whenever $\alpha>|\beta|$.

Let $F\into A\onto G$ be a short exact sequence in $\mathcal{A}^{\alpha,\obeta}$ for every $\alpha\gg0$; Lemma \ref{vert-lem-1} implies that $\calh^{-2}(G)=0$ and $\dim\calh^{0}(F)\le2$. As in the proof of Theorem \ref{vert-prop-1}, we must consider 3 situations: $\ch_0(F)=0$, $\ch_0(F)=-1$, and $\ch_0(F)=-2$. 

If $\ch_0(F)=0$, then $\tilde{F}=0$, and following the of Theorem \ref{vert-prop-1} when $\beta\ge\mu(A)=0$, it is enough to consider the case when $\ch_1(F)=0$ as well. It follows that $\tilde{G}$ is torsion free and we have the following exact sequence in $\coh(X)$:
$$ 0 \to 2\cdot\op3 \to \tilde{G} \to F_{00} \stackrel{u}{\to} \mathcal{O}_{L}(1) \to G_{00} \to 0. $$
If $u=0$, then $\tilde{G}=2\cdot\op3\oplus F_{00}$ because ${\rm Ext}^1(F_{00},2\cdot\op3)=H^2(F_{00}(-4))^{\oplus2}=0$ (because $\dim F_{00}=1$), contradicting the fact that $\tilde{G}$ must be torsion free. It follows that $u$ must be a monomorphism, since $\ker u\ne0$ would similarly contradict the torsion freeness of $\tilde{G}$; thus $\tilde{G}=2\cdot\op3$ and  $F=F_{00}=\mathcal{O}_L(1-z)$ for $z\ge1$, so that  
$$ \lim_{\alpha\to\infty} \left( \lambda_{\alpha,\overline{\beta},s}(F)-\lambda_{\alpha,\overline{\beta},s}(A) \right) =
\dfrac{6s-2}{3}\obeta - z < 0. $$

Next, we consider the case $\ch_0(F)=-2$, so that $\ch_0(\tilde{G})=0$. As in the proof of Theorem \ref{vert-prop-1}, it is enough to consider the case when $G=G_{00}$; in the case at hand, either $\dim G_{00}=0$ or $G_{00}=\mathcal{O}_L(1)$, in which case
$$ \lim_{\alpha\to\infty} \left( \lambda_{\alpha,\overline{\beta},s}(G)-\lambda_{\alpha,\overline{\beta},s}(A) \right) =
\dfrac{6s-2}{3}\obeta > 0, $$
as desired.

Finally, let $F\into A$ be a sub-object with $\ch_0(F)=-1$ and $\mu(F)=0$. It follows that $\tilde{F}=I_C$, the ideal sheaf of a codimension $\ge2$ subscheme $C\subset\p3$, so $2\cdot\op3/\tilde{F}=\op3\oplus\mathcal{O}_C$; but this is a subsheaf of $\tilde{G}$ which, as mentioned in the proof of Theorem \ref{vert-prop-1} for $\obeta\ge\mu(A)$, can only have torsion in codimension 1, thus $C=\emptyset$ and $\tilde{F}=\op3$. In addition, we can also conclude that the morphism $\tilde{G}\to F_{00}$ must vanish ($\ker u\ne0$ would contradict the fact that any torsion of $\tilde{G}$ must have dimension equal to 2), and hence $\tilde{G}\simeq\op3$ as well. In addition, it also follows that $F_{00}\simeq\mathcal{O}_{L}(1-z)$ for some $z>0$, and 
$$ \lim_{\alpha\to\infty} \alpha^2\left( \lambda_{\alpha,\overline{\beta},s}(F)-\lambda_{\alpha,\overline{\beta},s}(A) \right) =
\dfrac{6s-2}{3}\obeta - 2z < 0. $$

\hfill\qed
\end{example}




\section{Rank 2 instanton sheaves}\label{sec:instantons}

Recall that a torsion free sheaf $E$ on $\p3$ is called an \emph{instanton sheaf} if $\ch_1(E)=0$ and
\begin{equation}\label{instcond}
h^0(E(-1))=h^1(E(-2))=h^2(E(-2))=h^3(E(-3))=0.
\end{equation}
The positive integer $c:=-\ch_2(E)$ is called the \emph{charge} of $E$. Equivalently, there is a complex $M$ of the form
\begin{equation}\label{monad}
 c\,\op3(-1) \stackrel{a}{\rightarrow} (r+2c)\op3 \stackrel{b}{\rightarrow} c\,\op3(1),
\end{equation}
with $r$ being the rank of $E$, such that $\calh^{-1}(M)=\calh^{1}(M)=0$ and $\calh^0(M)\simeq E$; the complex $M$ is called a \emph{monad} for $E$. Note that $\ch(E)=(2,0,-c,0)$. Further details on the definitions and equivalences mentioned in this paragraph can be found in \cite{J-i}.

One can show that if $E$ is a nontrivial rank 2 instanton sheaf, then $E^{**}$ is a (possible trivial) locally free instanton sheaf and $Q_E:=E^{**}/E$ has pure dimension 1 \cite[Main Theorem]{GJ}. In addition, $E$ is Gieseker stable, being $E$ $\mu$-stable precisely when $E^{**}$ is nontrivial \cite[Theorem 4]{JMT}, and the sheaf $Q_E$ is semistable \cite[Lemma 6]{JMT}.

\begin{proposition}\label{inst-prop}
If $E$ is an instanton sheaf, then $A:=E[1]$ is asymptotically $\labs$-stable along any vertical line $\{\beta=\obeta\}$ for every $\obeta<0$ when $s>1/3$.
\end{proposition}

\begin{proof}
If $E$ is locally free, then $E$ is $\mu$-stable, so $A$ is asymptotically $\labs$-stable for every $s>0$, according to Example \ref{refl mu-st}. 

If $E$ is not locally free and $E^{**}$ is nontrivial, it is enough to check condition (4) in Theorem \ref{vert-prop-1}, since $E$ is $\mu$-stable.

Indeed, let $R$ be a subsheaf of $Q_E$ with $\ch(R)=(0,0,d,x)$ for some $d>0$. Since $Q_E$ is semistable with $\chi(Q_E(t))=k\cdot(t+2)$, we have
$$ \dfrac{\chi(R)}{\deg(R)} = \dfrac{2d+x}{d} \le 2 = \dfrac{\chi(Q_E)}{\deg(Q_E)}, $$
thus $x\le0$; it follows that
$$ \dfrac{\ch_3(R)}{\ch_2(R)} \le 0 < \dfrac{2-6s}{3}\obeta, $$
with the second inequality holding when $\obeta<0$ and $s>1/3$.

Finally, let us consider the case when $E$ is not locally free and $E^{**}$ is trivial, so that one cannot apply Theorem \ref{vert-prop-1} directly since $E$ is strictly $\mu$-semistable. However, considering an exact sequence $F\into E[1]\onto G$ within $\cohab$ for $\alpha\gg0$, we can follow the strategy of the proof of Theorem \ref{vert-prop-1} and consider three separate cases: $\ch_0(F)=-2$, $\ch_0(F)=-1$ and $\ch_0(F)=0$. 

By Lemma \ref{vert-lem-3}, exact sequence $F\into E[1]\onto G$ satisfies, for $\alpha\gg0$, $\calh^{-2}(F)=\calh^{-2}(G)=0$, $\dim\calh^0(F)\le2$ and $\calh^{-1}(F)$ and $\calh^{-1}(G)$ are both torsion free sheaves; in addition, $\calh^0(G)=0$ in this case. If $\ch_0(F)=-2$, then $\calh^{-1}(G)$ cannot be torsion free, while the hypothesis $\ch_0(F)=0$ leads to the same inequality of the item (4) in Theorem \ref{vert-prop-1}, which we already checked.

Therefore, it remains for us to check the case when $\ch_0(F)=\ch_0(G)=-1$ and $\mu(F)=\mu(G)=0$, so both $\calh^{-1}(F)$ and $\calh^{-1}(G)$ are ideal sheaves. We end up with the following exact sequence of Sheaves
$$ 0 \to I_C \to E \stackrel{f}{\to} I_Z \to P \to 0, $$
where $I_C:=\calh^{-1}(F)$, $I_Z:=\calh^{-1}(G)$ and $P:=\calh^0(F)$. Since $E$ is an instanton sheaf, the proof of \cite[Theorem 4]{JMT} implies that $C$ must pure dimension 1 and $\im f$ is the ideal sheaf of a (possibly empty) zero dimensional scheme, thus $\ch_2(I_C)=\ch_2(E)=-c$. Since $\im f \into I_Z$, we have that $\dim Z=0$, thus also $\dim P=0$ if $P\ne0$. Letting $z$ and $p$ be the lengths of $Z$ and $P$, respectively, we have $0\le p\le z$ (because $z-p$ is the length of $\op3/\im f$) and 
$$ \ch(\calh^{-1}(F))=(1,0,-c,z) ~~{\rm and}~~ \ch(\calh^{0}(F))=(0,0,0,p). $$
It follows that $\ch(F)=(-1,0,c,p-z)$, thus every sub-object of $E[1]$ within $\mathcal{A}^{\alpha,\obeta}$ with $\alpha\gg0$ and $\obeta\le0$ satisfies
$$ \delta_{20}(F,A)=-c<0 ~~{\rm and}~~ \delta_{30}(F,A)=2(z-p)\ge0, $$
thus 
$$ -\dfrac{3s-1}{3}\delta_{20}(F,A)\obeta + \dfrac{1}{2}\delta_{20}(F,A) - \dfrac{1}{2}\delta_{30}(F,A) < 0 $$
when $\obeta\le0$ and $s\ge1/3$.
\end{proof}

\begin{remark}\label{rem-inst}
The following two claims also follow from the proof of Proposition \ref{inst-prop}:
\begin{enumerate}
\item When $s<1/3$, there exists $\epsilon>0$ such that for every nontrivial instanton sheaf, the object $E[1]$ is asymptotically $\labs$-stable along a vertical line $\{\beta=\obeta\}$ for $-\epsilon<\obeta<0$. 
\item If $E$ is a non locally free instanton sheaf, then $A=E[1]$ is strictly asymptotically $\labs$-semistable along the vertical axis $\{\beta=0\}$ for every $s>0$; a destabilizing sequence is precisely $Q_E \to E[1] \to E^{**}[1]$. Note that the vertical axis, which coincides with $\Gamma_{v,s}^0$, is itself a numerical $\lambda$-wall.\hfill\qed
\end{enumerate}
\end{remark}

On the other hand, we show that there is an open subset $V\subset\HH$ containing the point $(\alpha=1/\sqrt{6s+1},\beta=0)$ in its boundary such that $E[1]$ is $\labs$-stable for every $(\alpha,\beta)\in V$ and every $s>0$.

To see this, the starting point is the fact that the monad \eqref{monad} can be broken down into two triangles
$$ c\,\op3(1) \to K[1] \to (2+2c)\op3[1] ~~{\rm and}~~ 
K[1] \to E[1] \to c\,\op3(-1)[2], $$
where $K:=\ker b$; note that $w:=\ch(K[1])=(-2-c,c,c/2,c/6)$. Setting $u:=\ch(\op3(1))$ and $v:=\ch(E[1])$, we obtain numerical $\lambda$-walls for $K[1]$ and $E[1]$ respectively denoted by $\Upsilon_{u,w,s}$ and $\Upsilon_{w,v,s}$, which intersect at the point $(\alpha=1/\sqrt{6s+1},\beta=0)$.
Observe that the objects $\op3(-1)[2]$, $\op3[1]$ and $\op3(1)$ do belong to $\cohab$ in an open neighbourhood of the point $(\alpha=1/\sqrt{6s+1},\beta=0)$ for every $s>0$, and therefore the objects $K[1]$ and $E[1]$ also belong to $\cohab$ in the same neighbourhood. It also follows that both $K[1]$ and $E[1]$ are $\labs$-semistable at the point $(\alpha=1/\sqrt{6s+1},\beta=0)$ for every $s>0$. Consequently, there is an open set $V\subset\HH$, bounded the $\alpha$-axis and by the numerical $\lambda$-wall $\Upsilon_{w,v,s}$, such that $E[1]$ is $\labs$-stable for every $(\alpha,\beta)\in V$, as desired. 

\begin{figure}
    \centering
    \includegraphics{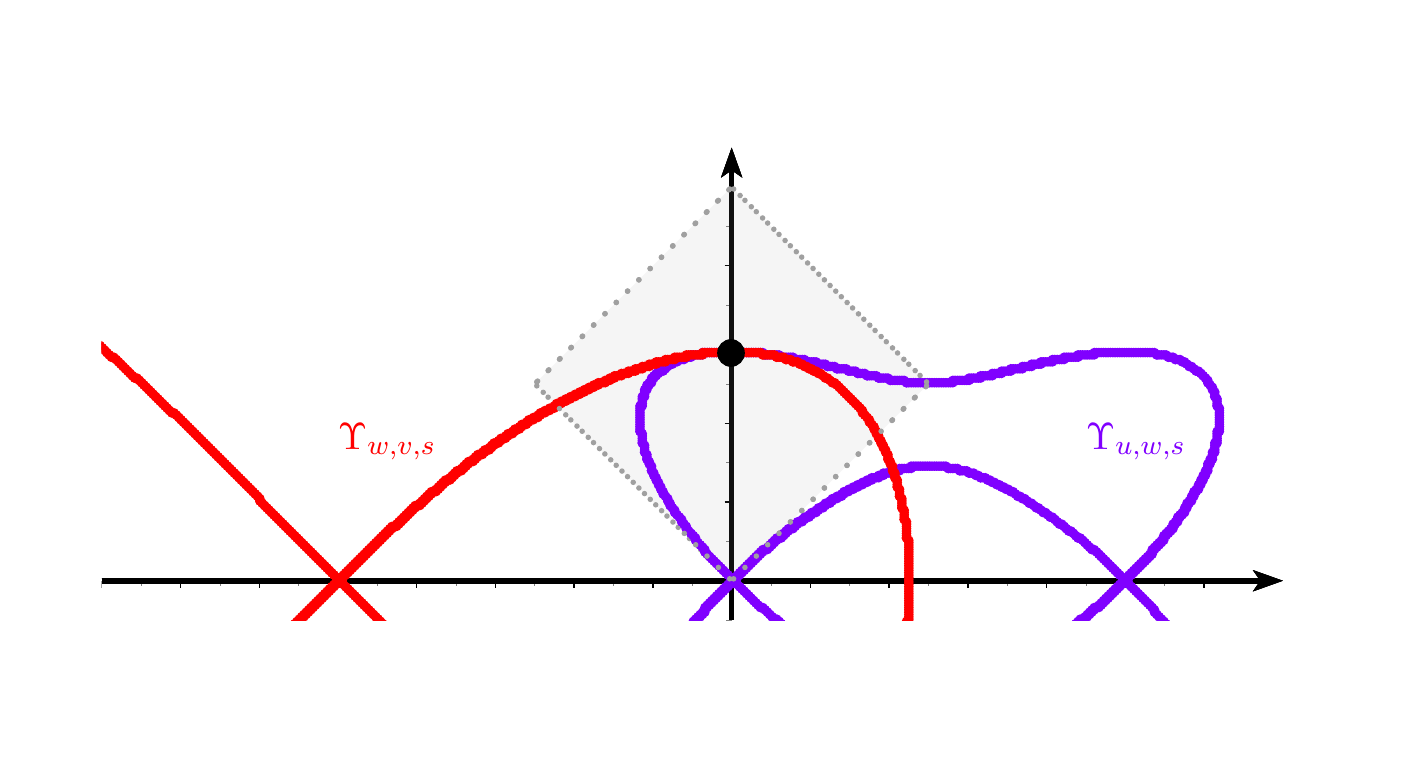}
    \caption{This graph shows the numerical $\lambda$-walls $\Upsilon_{u,w,s}$ (in purple) and $\Upsilon_{w,v,s}$ (in red) crossing at the point $(\alpha=1/\sqrt{3},\beta=0)$; we set $s=1/3$ and $c=2$. The shaded region marks where the objects $\op3(-1)[2]$, $\op3[1]$ and $\op3(1)$, and consequently also $K[1]$ and $E[1]$, belong to $\cohab$.}
    \label{fig:my_label}
\end{figure}

It is tempting to conjecture that there exists an $\epsilon>0$ such that every instanton sheaf is $\labs$-stable for $\alpha>\sqrt{6s+1}+\epsilon$ and $-\epsilon\le\beta\le0$ and every $s>0$. This statement will verified when $\beta=0$ for the cases $c=1,2$, in Proposition \ref{LSSO_(-2,0,1,0)}, Example \ref{walls-charge1}, and Example \ref{walls-charge2} below.


\begin{remark}\label{monad-wall}
For the curve $\{\beta=0\}$, the $\lambda$-wall produced by the subobject $K[1]$, which destabilizes all shifts of instanton sheaves, is given by the equation
$$
(6s+1)\alpha^2=1.
$$
Crossing this wall does not produce any new $\lambda$-semistable flipped objects since $$
\Ext^1(K[1],\mathcal{O}_{\mathbb{P}^3}(-1)[2])=0.
$$ 
For charge 1, we will see later that, indeed, the last step is vacuous. However, the strategy works in greater generality: one must analyse the type of object which arises, and classify walls for that type. By local finiteness, this process will terminate after finitely many steps.
\end{remark}

Let us begin with the classification of asymptotically $\labs$-semistable objects along $\{\beta=0\}$. First, let us recall the classification of $\mu$-semistable sheaves of Chern character $(2,0,-1,0)$:

\begin{proposition}\cite[Proposition 8.2]{JM19}\label{(2,0,-1,0)}
Let $E$ be a $\mu$-semistable torsion free sheaf on $X=\p3$ with $\ch(E)=(2,0,-1,0)$.
\begin{enumerate}
\item If $E$ is locally free, then $E$ is a null correlation bundle; in particular, $E$ is $\mu$-stable.
\item If $E$ is properly torsion free, then $E$ is strictly $\mu$-semistable and it is given by one of the following extensions:
\begin{itemize}
\item[(2.1)] $0\to I_{L} \to E \to I_{p} \to 0$ for $p\in L$ with nontrivial extension; in particular, $E$ is a null correlation sheaf and it is $\mu_{\le2}$-stable;
\item[(2.2)] $0\to I_{p} \to E \to I_{L} \to 0$ for arbitrary $p$ and $L$; in particular, $E$ is not $\mu_{\le2}$-semistable and it has no global sections;
\item[(2.3)] $0\to \op3 \to E \to I_{\tilde{L}} \to 0$, where $\tilde{L}$ is a 1-dimensional scheme satisfying the sequence $0\to\mathcal{O}_p\to\mathcal{O}_{\tilde{L}}\to\mathcal{O}_L\to0$ for arbitrary $p$ and $L$; in particular, $E$ is not $\mu_{\le2}$-semistable and it has a global section.
\end{itemize}
\end{enumerate}
\end{proposition}

Notice that if $E$ is a sheaf of type (2.2) and (2.3) in Proposition \ref{(2,0,-1,0)}, then $E[1]$ is not $\labs$-semistable when $(\alpha,\beta)\in R^0_v$ and $s>0$, for they are destabilised by $\mathcal{O}_p\into E[1]$. Moreover, a sheaf of type (2.3) cannot be $\labs$-semistable for $(\alpha,\beta)\in R_v^-$ and $s>0$ either, since  
$$ \labs(E) - \labs(\op3) = \dfrac{((6s-2)\alpha^2+\beta^2)\beta}{\ch_2^{\alpha,\beta}(E)\ch_2^{\alpha,\beta}(\op3)} < 0. $$
In addition, this same inequality shows that the object $A$ obtained from flipping the exact sequence $\mathcal{O}_p\into E[1]\onto S_L[1]$, then $\op3[1]$ is also a $\labs$-destabilising sub-object for $A$ when $(\alpha,\beta)\in R^0_v$.

On the other hand, we have already seen in Section \ref{sec:instantons} that the sheaves of type (1) and (2.1) in Proposition \ref{(2,0,-1,0)}, which are the instanton sheaves of charge 1, are $\lambda_{1/\sqrt{6s+1},0,s}$-semistable for every $s>0$.

We will now show that every asymptotically $\labs$-semistable object along $\{\beta=0\}$ is either the shift of an instanton sheaf of charge 1 or of its dual.

\begin{proposition}\label{LSSO_(-2,0,1,0)}
An object $A\in D^{\rm b}(\p3)$ with $\ch(A)=(-2,0,1,0)$ is asymptotically $\labs$-semistable along $\{\beta=0\}$ if and only if
\begin{enumerate}
\item $A=E[1]$, where $E$ is a rank 2 instanton sheaf of charge 1; or
\item $A=E^\vee[1]$, where $E$ is a rank 2 properly torsion free instanton sheaf of charge 1.
\end{enumerate}
In addition, $A$ is asymptotically $\labs$-stable along $\{\beta=0\}$ if and only if $A=E[1]$, where $E$ is a rank 2 locally free instanton sheaf of charge 1.
\end{proposition}
\begin{proof}
We have already seen in Section \ref{sec:instantons} that objects of type (1) above are asymptotically $\labs$-semistable along $\{\beta=0\}$, and that $A$ is asymptotically $\labs$-stable along $\{\beta=0\}$ only when $E$ is locally free. 

If $E$ a properly torsion free rank 2 instanton sheaf of charge 1, then it satisfies the triangle $\mathcal{O}_L(1)\to E[1] \to 2\cdot\op3[1]$; $E^\vee[1]$ satisfies the flipped triangle, and these were shown to be asymptotically $\labs$-semistable along $\{\beta=0\}$ in Example \ref{example dual instanton} above.

Conversely, assume that $A$ is asymptotically $\labs$-semistable object along $\{\beta=0\}$; we set $E:=\mathcal{H}^{-1}(A)$ and $T:=\mathcal{H}^{0}(A)$, with $\ch(E)=(2,0,-a,b)$ and $\ch(T)=\ch(A)+\ch(E)=(0,0,1-a,b)$. Note that $a\ge0$, because $E$ is $\mu$-semistable, while $1-a\ge0$, thus either $a=0$ or $a=1$. 

Consider fist the case $a=0$. $E^{**}$ is $\mu$-semistable, thus $\ch_2(E^{**})\le0$; however, $\ch_2(E^{**})=\ch_2(E^{**}/E)\ge0$, thus $\ch_2(E^{**})=0$. But $E^{**}/E$ has pure dimension 1, hence $E\simeq E^{**}=2\cdot\op3$, since the only $\mu$-semistable rank 2 reflexive sheaf with $\ch_1=\ch_2=0$ is the trivial one. It follows that $\ch(T)=(0,0,1,0)$, thus $T=\mathcal{O}_L(1)$, giving the objects of type (2).

Next, assume $a=1$, so $\ch(E)=(2,0,-1,b)$ and $\ch(T)=(0,0,0,b)$ so $b\ge0$. 

If $E$ is reflexive, then $b=0,1$. If $b=0$, then $E$ must be a locally free instanton sheaf of charge 1, as mentioned in the proof of \cite[Proposition 8.2]{JM19}. If $b=1$, then $E=S_L$ is a nontrivial extension
$$
0\rightarrow \mathcal{O}_{\mathbb{P}^3}\rightarrow S_L\rightarrow I_L\rightarrow 0,
$$
for some line $L$, and $T=\mathcal{O}_p$. 

However, such an object, which is defined by the exact sequence in $\mathcal{A}^{\alpha,0}$
\begin{equation}\label{slop}
0 \to S_L[1] \to A \to \mathcal{O}_p \to 0
\end{equation}
is not asymptotically $\labs$-semistable along $\{\beta=0\}$: there exists an epimorphism $A\onto I_L[1]$ whose kernel is $\op3[1]\oplus\mathcal{O}_p$;
since 
$$ \lambda_{\alpha,0,s}(A) - \lambda_{\alpha,0,s}(I_L[1]) = \dfrac{2}{\alpha^2+1}>0, $$
we conclude that the object given by the sequence in display \eqref{slop} is never $\labs$-semistable along the $\alpha$-axis.

Finally, if $E$ is not reflexive, then let $Q_E:=E^{**}/E$; this is a sheaf of pure dimension 1, hence $1\le\ch_2(Q_E)=\ch_2(E^{**})+1\le1$, since $\ch_2(E^{**})\le0$ by the Bogomolov inequality. It follows that $E^{**}=2\cdot\op3$ and $\ch(Q_E)=(0,0,1,-b)$, thus $Q_E=\mathcal{O}_L(-b+1)$ for some line $L$. But $h^0(Q_E)>0$, so $-b+1\ge0$; recalling that $b\ge0$, we again must have $b=0,1$.

If $b=1$, then $E=\op3\oplus I_L$, and it is clear that the object $A$ given by the exact sequence
$$ 0 \to \op3[1]\oplus I_L[1] \to A \to \mathcal{O}_p \to 0 $$
cannot be asymptotically $\labs$-semistable along the $\alpha$-axis, as it is also destabilised by the epimorphism $A\onto I_L[1]$. 

We are left with the case $b=0$, which forces $E$ to be a properly torsion free instanton sheaf, since it satisfies the exact sequence $0\to E\to2\cdot\op3\to\mathcal{O}_L(1)\to0$ in $\coh(X)$. 
\end{proof}

Similar computations as in Proposition \ref{LSSO_(-2,0,1,0)} allow us to describe all the asymptotically $\lambda_{\alpha,\beta,s}$-semistable objects of Chern character $(-2,0,2,0)$ along $\{\beta=0\}$. More precisely, we have the following:
\begin{proposition}
An object $A\in D^b(\mathbb{P}^3)$ with $\ch(A)=(-2,0,2,0)$ is asymptotically  $\lambda_{\alpha,\beta,s}$-semistable along $\{\beta=0\}$ if and only if
\begin{enumerate}
    \item $A=E[1]$, where $E$ is a rank 2 instanton sheaf of charge 2; such objects are limit stable if and only if $E$ is locally free.
    \item $A=E^{\vee}[1]$, where $E$ is a non locally free rank 2 instanton sheaf of charge 2; all of these objects are strictly limit semistable, and are S-equivalent to $E[1]$.
    \item $A$ fits into an exact sequence 
    $$ 0\rightarrow F[1]\rightarrow A\rightarrow \mathcal{O}_{L}(2)\rightarrow 0, $$
    where $F$ is a strictly $\mu$-semistable sheaf with $\ch(F)=(2,0,-1,1)$; all of these objects are strictly limit semistable.
    \item $A$ fits into an exact sequence 
    $$
    0\rightarrow S[1]\rightarrow A\rightarrow Z\rightarrow 0,
    $$
    where $S$ is a $\mu$-semistable reflexive sheaf with $\ch(S)=(2,0,-2,z)$ and $Z$ is a 0-dimensional sheaf of length $z$; such an object is strictly limit semistable if and only if $S$ is strictly $\mu$-semistable.
\end{enumerate}
\end{proposition}
\begin{proof}
Assume that $A$ is limit semistable; from Theorem \ref{vert-prop-2} we know that $A$ fits into an exact sequence in $\mathcal{A}^{\alpha,0}$ of the form
\begin{equation}\label{type0}
0\rightarrow E[1]\rightarrow A\rightarrow T\rightarrow 0,
\end{equation}
where $E$ is a $\mu$-semistable sheaf and $T$ is a torsion sheaf supported in dimension $\leq 1$. Let $\ch(T)=(0,0,a,b)$, then $\ch(E)=(2,0,a-2,b)$ and since $E$ satisfies the Bogomolov inequality we must have $0\leq a\leq 2$. Recall also that $Q_E:=E^{**}/E$ has pure dimension 1. 

{\bf Case $a=2$}. In this case, $\ch_2(E)=0$ and since $E^{**}$ is also $\mu$-semistable then $\ch_2(E^{**})\leq 0$. This implies that $E$ is reflexive, since otherwise $E^{**}/E$ would be supported in dimension 1 by Theorem \ref{vert-prop-2}. Since the only rank 2 reflexive sheaf with $\ch_1(E)=\ch_2(E)=0$ is trivial, we conclude that $E=2\cdot\op3$. It follows that $A$ fits into an exact sequence
$$ 0\rightarrow 2\cdot\op3[1]\rightarrow A\rightarrow T \rightarrow 0
$$
with $T$ being a semistable 1-dimensional sheaf with $\ch(T)=(0,0,2,0)$; according to \cite[Prop 18]{JMT}, either $T=\mathcal{O}_{L'}(1)\oplus\mathcal{O}_{L''}(1)$, where $L'$ and $L''$ are lines (potentially coincident), or  $T=\mathcal{O}_{C}(3{\rm pt})$, where $C$ is a conic. In both cases, $A=E^\vee[1]$ where $E$ is a non locally free instanton sheaf of charge 2. Since $\lambda_{\alpha,0,s}(\op3[1])=\lambda_{\alpha,0,s}(T)=0$, every such object is strictly limit semistable.

{\bf Case $a=1$}. Firstly, suppose that $E$ is reflexive, then $b=0,1$. If $b=0$, then $E$ is a rank 2 locally free instanton sheaf of charge 1 (see \cite[Lemma 2.1]{Chang}) and $T=\mathcal{O}_L(1)$ for some line $L$. If $b=1$, then $E$ is strictly $\mu$-semistable and, also according to \cite[Lemma 2.1]{Chang}, it fits into an exact sequence of sheaves
$$ 0\rightarrow \mathcal{O}\rightarrow E\rightarrow I_{L'}\rightarrow 0, $$
for some line $L'$. Thus $A$ fits into an exact sequence 
\begin{equation}\label{type3}
0 \rightarrow E[1] \rightarrow A \rightarrow \mathcal{O}_{L''}(2) \rightarrow 0,
\end{equation}
for some other line $L''$.

If $E$ is not reflexive, then $Q_E=E^{**}/E$ is pure 1-dimensional and so $\ch_2(Q_E)\geq 1$. Notice that
$$
\ch_2(E^{**})+1=\ch_2(E^{**})-\ch_2(E)=\ch_2(Q_E)\geq 1,
$$
and so $E^{**}=2\cdot\op3$, since $E^{**}$ reflexive and $\mu$-semistable with $\ch_2(E^{**})=0$. Therefore, $Q_E=\mathcal{O}_L(1-b)$ for some line $L$ and moreover $Q_E$ has a section and so $1-b\geq 0$. Thus $b=0,1$ because of part (2) of Theorem \ref{vert-prop-2}; 

\begin{itemize}
\item if $b=1$, then $E=\op3\oplus I_{L'}$, for some line $L'$, and $A$ fits into an exact sequence like the one in display \eqref{type3};
\item if $b=0$, then $Q_E=\mathcal{O}_L(1)$, which implies that $E$ is a properly torsion-free instanton sheaf of charge 1, and $A$ fits into an exact sequence
$$
0\rightarrow E[1] \rightarrow A\rightarrow \mathcal{O}_{L'}(1)\rightarrow 0;
$$
again, $A=F^\vee[1]$, where $F$ is a non locally free rank 2 instanton sheaf of charge 2; note that these objects are strictly limit semistable, since both $E[1]$ and $\mathcal{O}_{L'}(1)$ are limit stable and have the same slope.
\end{itemize}

Finally, we check that objects given by the short exact sequence in display \eqref{type3} are strictly limit semistable. Indeed, such objects also fit in the sequence
$$ 0\to \op3[1] \to A \to G \to 0 ~,~ \textrm{ where }~ 0 \to I_{L'}[1] \to G \to \mathcal{O}_{L''}(2) \to 0; $$
since both $\op3[1]$ and $G$ are limit stable and have the same slope, $A$ is strictly limit semistable.


{\bf Case $a=0$}. In this case, $\ch(T)=(0,0,0,b)$, so $T$ is a 0-dimensional sheaf of length $b\ge0$.

If $E$ is locally free, then
$$ \Ext^1(T,E[1]) \simeq \Ext^1(E,T)^* \simeq H^1(E^*\otimes T)^* = 0, $$
so exact sequences as in display \eqref{type0} split and therefore are not limit semistable if $b\ne0$. Therefore, these objects are of type (1) and are limit stable.

If $E$ is reflexive, then an object of the form
$$ 0\rightarrow E[1]\rightarrow A\rightarrow T\rightarrow 0,$$
such that no subsheaf $U\hookrightarrow T$ factors through $A$, satisfies all of the conditions of Theorem \ref{vert-prop-2}. To see that such objects are limit (semi)stable if and only if $E$ is $\mu$-(semi)stable, see Lemma \ref{S[1]AZ} below.

If $E$ is not reflexive, then $E$ is a non locally-free instanton sheaf of charge 2. It then follows that $\Ext^1(T,E[1])=0$, so such objects are limit semistable if and only if $b=0$. In addition, these objects are of type (1) and are limit semistable, see Remark \ref{rem-inst}.

\end{proof}

\begin{lemma}\label{S[1]AZ}
Let $S$ be a rank 2 reflexive sheaf on $\mathbb{P}^3$ with $\mu(S)=0$, and let $Z$ be a 0-dimensional sheaf with $h^0(Z)=\ch_3(S)$. An object $A$ defined by the triangle
$$ S[1] \rightarrow A \rightarrow Z $$
such that no subsheaf $U\to Z$ lifts to a sub-object of $A$ is asymptotically $\lambda_{\alpha,0,s}$-(semi)stable if and only if $S$ is $\mu$-(semi)stable
\end{lemma}
\begin{proof}
Let $F\hookrightarrow A$ be a sub-object for $\alpha\gg0$ with $\ch(F)=(-r,c,d,e)$. Let $F'$ be the kernel of the composite morphism $F\hookrightarrow A\onto Z$, where $Z=\mathcal{H}^{0}(A)$; note that $F'$ is a sub-object of $S[1]$, and let $G':=S[1]/F'$. We get the following exact sequence of sheaves
$$ 0 \to \mathcal{H}^{-1}(F') \to S \to \mathcal{H}^{-1}(G') \to \mathcal{H}^{0}(F') \to 0 $$
and $\mathcal{H}^{0}(G')=0$. Moreover, $\mathcal{H}^{-1}(F')\simeq\mathcal{H}^{-1}(F)$ because $F/F'$ is a subsheaf of $Z$, thus $\ch_0(F')=r\le2$, $\ch_1(F')=c$  and $\ch_2(F')=d$. In addition, since $\mathcal{H}^{-1}(G')$ is torsion free and $S$ is reflexive, it follows that $\mathcal{H}^{-1}(F')$ is also reflexive.

If $r=2$, then $\mathcal{H}^{-1}(F')=S$ and $F$ is given by an exact sequence of the form
$$ 0 \to S[1] \to F \to U \to 0 $$
where $U$ is a subsheaf of $Z$.
It follows that $\ch_2(F)=\ch_2(A)$, and $\ch_3(F)=-\ch_3(S)+h^0(U)\le\ch_3(A)$, with the equality only occurring only when $F=A$. Therefore,
$$ \lambda_{\alpha,0,s}(A) - \lambda_{\alpha,0,s}(F) = \dfrac{\ch_3(A)-\ch_3(F)}{\ch_2(A)+\alpha^2}>0. $$

If $r=1$, then $\mathcal{H}^{-1}(F')=\op3(c)$, $F$ must be given by an exact sequence of the form
$$ 0 \to \op3(c)[1] \to F \to U \to 0 $$
where $U$ is a subsheaf of $Z$. However, $\Ext^1(U,\op3(c)[1])=0$, thus if $U\ne0$, then $U$ is a subsheaf of $Z$ lifting to a sub-object of $A$, contradicting our hypothesis. If follows that $U=0$, and $F=\op3(c)[1]$; one can check that 
$$ \lim_{\alpha\to\infty} \big( \lambda_{\alpha,0,s}(A) - \lambda_{\alpha,0,s}(F) \big) = -2(s+1/6)c. $$

Finally, if $r=0$, then $F$ is a subsheaf of $Z$ lifting to a sub-object of $A$, contradicting the hypothesis.

Therefore, if $A$ is asymptotically $\lambda_{\alpha,0,s}$-(semi)stable, then $c<0$ ($c\le0$), implying that $S$ is $\mu$-(semi)stable.

Conversely, if $S$ is $\mu$-stable, then $c<0$, and $A$ is asymptotically $\lambda_{\alpha,0,s}$-stable.

If $S$ is strictly $\mu$-semistable, then $\op3[1]$ is the only possible destabilizing sub-object of $A$; since $\lambda_{\alpha,0,s}(A) = \lambda_{\alpha,0,s}(\op3[1])=0$, we conclude that $A$ is strictly $\lambda_{\alpha,0,s}$-semistable.
\end{proof}

In section \ref{sec:beta=0}, we will classify all actual $\lambda$-walls for $v=(-2,0,1,0)$ and $v=(-2,0,2,0)$ when $\beta=0$.
\begin{confidential}
\begin{proposition}\label{gamma0inst}
Let $v=(-2,0,1,0)$ and assume $s>1/6$. There are precisely two actual $\lambda$-walls for $v$ intersecting the $\alpha$-axis, both of them at the point $(\alpha=1/\sqrt{6s+1},\beta=0)$, given by the sequences
\begin{enumerate}
\item the \emph{monad} wall $0\to\Omega^1_{\p3}(1)[1] \to A \to \op3(-1)[2] \to 0$;
\item the \emph{dual} wall $0\to\op3(1) \to A \to T\p3(-1)[1] \to 0$.
\end{enumerate}
The first wall is a vanishing $\lambda$-wall for $\beta\le0$ while the second one is a vanishing $\lambda$-wall for $\beta\ge0$ when $s>1/3$; the situation reverses when $s\le1/3$. 
\end{proposition}
\end{confidential}

\begin{confidential}
\begin{proof}
By the considerations above, we may assume that $A:=E[1]$, where $E$ is an instanton sheaf of charge 1.
Then $E[1]\in\coh^0(\mathbb{P}^3)\cap\mathcal{A}^{\alpha,0}$, $\lambda_{\alpha,0,s}(E)=0$ for all $\alpha$ and $s$, and $\ch^{\alpha,0}(E[1])=1+\alpha^2$. Suppose 
\begin{equation}\label{destab_seq}
0\to A\to E[1]\to B\to 0
\end{equation}
is a short exact sequence in $\mathcal{A}^{\alpha,0}$ for some $\alpha$ such that $\lambda_{\alpha,0,s}(A)=0$. We may assume $A$ is $\lambda_{\alpha,0,s}$-stable and $B$ is $\lambda_{\alpha,0,s}$-semistable. Let $\ch(A)=(r,x,y/2,z/6)$ for integers $r,x,y,z$. Then 
\begin{gather*}
\lambda_{\alpha,0,s}(A)=z/6-\alpha^2(s+1/6)x ~~ \text{and}\\
\ch_2^{\alpha,0}(A)=y/2-\alpha^2 r/2.
\end{gather*}
So $x=0$ implies $z=0$ and then $\beta=0$, which is the wall we do not want. We will now show that excluding this case results in one of three possible short exact sequences in $\coh(X)$:
\begin{equation}\label{sheaf_seq1}
0\to \mathcal{H}^{-2}(B)\to\mathcal{H}^{-1}(A)\to E\to 0
\end{equation}
or
\begin{equation}\label{sheaf_seq}
0\to E\to \mathcal{H}^{-1}(B)\to\mathcal{H}^0(A)\to 0
\end{equation}
or
\begin{equation}\label{sheaf_seq2}
0\to \mathcal{H}^{-1}(A)\to E\to\mathcal{H}^{-1}(B)\to 0,
\end{equation}
where in the last sequence $\ch_0(\mathcal{H}^{-1}(B))=0$.

To do this take cohomology in $\coh^0(\mathbb{P}^3)$ of the exact sequence in display (\ref{destab_seq}). We use superscripts for this cohomology:
\[0\to B^{-1}\to A^0\to E[1]\to B^0\to 0,\]
where $B^i=\mathcal{H}^{i}_{\coh^0(\mathbb{P}^3)}(B)$ etc.
If $B^0\neq0$, then $E[1]$ is $\nu_{\alpha,0}$-stable with $\nu_{\alpha,0}(E[1])=\infty$, thus $\nu_{\alpha,0}(B^0)=\infty$. But then $\ch_1(B^0)=0$. Factoring the sequence above via $F\in\coh^0(\mathbb{P}^3)$ we have $\lambda_{\alpha,0,s}(B^{-1})\leq0$ and $\lambda_{\alpha,0,s}(A^0)>0$ (note that $A^0=A$ and so is non-zero) so that $\lambda_{\alpha,0,s}(F)>0$ contradicting the fact that $E[1]$ is $\lambda_{\alpha,0,s}$-semistable at that value of $\alpha$, This implies that $B^0=0$. 
So $B[-1]\in\coh^0(\mathbb{P}^3)$ while $A\in\coh^0(\mathbb{P}^3)$. Now consider the long exact sequence in $\coh(X)$:
\[0\to \mathcal{H}^{-2}(B)\to\mathcal{H}^{-1}(A)\to E\to\mathcal{H}^{-1}(B)\to\mathcal{H}^0(A)\to 0.\]
But $\Hom(\mathcal{H}^{-1}(A)[1],B[-1])=0$ and so we have an octahedron
\[\xymatrix@=1.5pc{&\mathcal{H}^{-1}(A)[1]\ar[d]\ar@{=}[r]&\mathcal{H}^{-1}(A)[1]\ar[d]\\
B[-1]\ar[r]\ar@{=}[d]& A\ar[d]\ar[r]& E[1]\ar[d]\\
B[-1]\ar[r]&\mathcal{H}^0(A)\ar[r]& D }\]
for some object $D$. Then the bottom triangle implies that $D\in\coh^0(\mathbb{P}^3)$ then the right vertical triangle implies $\ch_1(\mathcal{H}^{-1}(A))=0$. So $\ch_1(B)=0=\ch_1(A)$ and then we have an unbounded wall. To avoid this, either $\mathcal{H}^0(A)=0$ or $\mathcal{H}^{-1}(A)=0$ ($D$=0 is not an option as then $E[1]\to B$ would be zero). If $\mathcal{H}^0(A)=0$ then the kernel of the sheaf map $E\to \mathcal{H}^{-1}(B)$ is also a destabilizer and either the rank or the first Chern class of $\mathcal{H}^{-1}(B)$ is zero. In the latter case we would have an unbounded wall unless $\mathcal{H}^{-1}(B)=0$ and then we obtain the sequence in display (\ref{sheaf_seq1}). In the former case, we get to the sequence in display (\ref{sheaf_seq2}) after we replace $A$ with its image in $E$ and then $B$ with the image of $E[1]\to B$. Otherwise we must have $\mathcal{H}^{-1}(A)=0$ and then $\mathcal{H}^{-2}(B)=0$ which establishes the sequence in display (\ref{sheaf_seq}).

We now show that we can choose $A$ and $B[-1]$ to be $\nu_0$-semistable. Suppose first that $A$ is not $\nu_0$-semistable. Then there is a short exact sequence $0\to K\to A\to Q\to0$ in $\coh^0(\mathbb{P}^3)$ with $\nu(K)>\nu(A)>\nu(Q)$. We may assume $K$ is $\nu$-stable. We now have a 3x3 diagram in $\coh^0(\mathbb{P}^3)$:
\[\xymatrix@=2pc{D\ar@{^{(}->}[r]\ar@{^{(}->}[d]&K\ar@{->>}[r]\ar@{^{(}->}[d]&G\ar@{^{(}->}[d]\\
B^{-1}\ar@{^{(}->}[r]\ar@{->>}[d]&A\ar@{->>}[r]\ar@{->>}[d]&E[1]\ar@{->>}[d]\\
I\ar@{^{(}->}[r]&Q\ar@{->>}[r]&F}\]
where $G\neq0$ because $\nu_0^+(B[-1])\leq0$ and $\nu_0(K)>0$. But now it follows that $D[1],I[1]\in\mathcal{A}^{\alpha,0}$. Then from the $\lambda_{\alpha,0}$-semistability of $A$, $B$ and $E[1]$ it follows that $G\to E[1]\to F$ would destabilize $E[1]$. This contradicts the above argument. So $F=0$. Then $K\to E[1]\to D$ would still destabilize $E[1]$.

On the other hand, if $B^{-1}$ is not $\nu_{0}$-semistable then we have a 3x3 diagram in $\coh^0(\mathbb{P}^3)$:
\[\xymatrix@=2pc{K\ar@{=}[r]\ar@{^{(}->}[d]&K\ar@{^{(}->}[d]\\
B^{-1}\ar@{^{(}->}[r]\ar@{->>}[d]&A\ar@{->>}[r]\ar@{->>}[d]&E[1]\ar@{=}[d]\\
Q\ar@{^{(}->}[r]&D\ar@{->>}[r]&E[1]}\]
This gives the long exact sequence in $\mathcal{A}^{\alpha,0}$:
\[0\to \mathcal{H}_\mathcal{A}^0(Q)\to D\to E[1]\to \mathcal{H}_{\mathcal{A}}^1(Q)\to0. \]
As before the $\lambda_{\alpha,0,s}$-slopes of these are all equal and then the argument above implies that $\mathcal{H}_{\mathcal{A}}^1(Q)=0$. But $\nu(Q)<0$ and so $Q\not\in\mathcal{A}$, a contradiction. 

We now look first for psuedo-walls. Then the constraints which $A$ and $B$ must satisfy are:
\begin{align}
\label{Abog} x^2\geq ry,\quad x^2\geq&(y-2)(r+2) &\text{Bogomolov inequalities}\\
\label{Arank}\alpha^2 r<y<2+&\alpha^2(r+2)& 0<\ch_2^{\alpha,0}(A)<\ch_2^{\alpha,0}(E[1])\\
\label{gen_bog1}\alpha^2(x^2-ry)+y^2&\geq xz&\text{generalized Bogomolov for $A$}\\
\label{gen_bog2}\alpha^2(x^2-(y-2)(r+2))+&(2-y)^2\geq xz &\text{generalized Bogomolov for $B[1]$}\\
\label{alpha_is}\alpha^2=\frac{z}{x(6s+1)}& &\text{$\lambda_{\alpha,0,s}$-destabilizing condition}
\end{align}
We now assume we are in situation of display (\ref{sheaf_seq}) above. The other cases (\ref{sheaf_seq1}) and (\ref{sheaf_seq2}) together is similar (by switching $A$ and $B$ and the signs of the Chern character or by dualizing) and we leave it to the reader.
Then $r\geq0$ and $x>0$ as $\mu_-(\mathcal{H}^{-1}(B))>0$. Then from (\ref{alpha_is}) we have $z>0$. 
Note that the wall corresponding to $u=-\ch(\ox(1))$ has $\alpha^2=1/(6s+1)$ as described above. For now assume that $\alpha^2\geq 1/(6s+1)$ so we are looking for walls above this \emph{monad} wall. We now show that $r\neq0$. Suppose $r=0$. We can re-arrange (\ref{gen_bog2}) to
\[(y-2)(2\alpha^2+2-y)\leq \alpha^2x^2-xz=x\alpha^2(x-x(6s+1))=-6s\alpha^2x^2<0.\]
From (\ref{Arank}), it follows that $0<y<2$ and so $y=1$. Then from (\ref{gen_bog1}) we have 
$1\geq 6s\alpha^2x^2$. It follows that $x^2\leq \dfrac{6s+1}{6s}<2$ when $s>1/6$. Then $x=1$ and $\alpha^2\leq 1/(6s)$. Finally, from (\ref{gen_bog1}) again $z\leq \alpha^2+1\leq 1+1/6s<2$ and so $z=1$. This gives a wall corresponding to $A=\mathcal{O}_H(1)$ from the twisted structure sheaf of a hyperplane. This exists because $\Ext^1(\mathcal{O}_H(1),E)\to \Ext^1(\ox(1),E)$ surjects and the latter group is one dimensional. The extension class gives a rank 2 reflexive sheaf $B[-1]=T\mathbb{P}^3(-1)/\ox\cong S_p(2)$, where $S_p$ is the sheaf defined by the exact triangle
$$
3\mathcal{O}_{\mathbb{P}^3}(-2)\rightarrow S_p\rightarrow \mathcal{O}_{\mathbb{P}^3}(-3)[1].
$$
This goes through $\alpha^2=1/(6s+1)$ and gives a well defined wall on one side of $\beta=0$ depending on the sign of $s-1/3$. More precisely, it is well defined for $\beta\leq0$ when $s\ge 1/3$.

Now we suppose that $r>0$. From (\ref{Arank}) we have $\alpha^2<y/r$ and combining with (\ref{gen_bog1}) we have 
\[\frac{y}{r}(x^2-ry)\geq xz-y^2. \]
This gives $xy\geq rz$. From (\ref{Abog}) we have $x^3\geq rxy\geq r^2z\geq r^2x$ since $\alpha^2\geq 1/(6s+1)$. Then $x\geq r$. 

Finally, (\ref{gen_bog2}) gives us
\[(y-2)\bigl(\alpha^2(r+2)-(y-2)\bigr)\leq -xz\left(\frac{6s}{6s+1}\right)\]
The second factor on the LHS is positive by (\ref{Arank}) so that we conclude that $0<y< 2$. So $y=1$. Then from $z\geq x\geq rz$ we deduce that $r=1$. Now from (\ref{gen_bog1}), we have $1-\alpha^2\geq xz\left(\dfrac{6s}{6s+1}\right)$ and then $\alpha^2\geq 1/(6s+1)$ gives $1\geq xz>0$ so that $x=z=1$ as well. This is precisely the instanton wall. 

Note that the dual of the rank zero wall gives a third solution in the form (\ref{sheaf_seq2}) with Chern character $(0,0,1,-1)$. But this does not correspond to an actual wall as the image of $E$ in $\mathcal{O}_H$ is $\mathcal{I}_{p/H}$ which can be seen by dualizing  $E[1]\to S_p(2)[1]\to \mathcal{O}_H(1)$:
\[0 \to S_p(1)\to E\to\mathcal{O}_H\to\mathcal{O}_p\to 0.\]
So there is no actual wall (although there is a numerical wall).
\end{proof}
\end{confidential}

\begin{confidential}
\begin{remark}
This proof shows that the only wall with $r>0$ is the monad wall for all $s>0$. But when $s$ is small, there may be further rank $0$ walls; these would correspond to torsion sheaf quotients $T$ of $\ox(1)$ for which $\Ext^1(E,T)\neq0$. These can exist, and would correspond to a vanishing wall above the monad wall.\hfill\qed
\end{remark}
\end{confidential}


\section{Bridgeland walls for $\ch^{\beta}=(-R,0,D,0)$}\label{sec:bridgeland_walls}

For a fixed value of the parameter $\beta$, we consider the two-dimensional cone of Bridgeland stability conditions 
$$
\sigma_{\alpha, \beta, s}=(Z_{\alpha,\beta,s},\mathcal{A}^{\alpha,\beta}),\   \alpha,s>0.
$$
A Bridgeland wall for the (twisted) Chern character vector $v=\ch^{\beta}=(-R,0,D,0)$, where $R\geq 0$ and $D> 0$, is produced by a short exact sequence
$$
0\rightarrow A\rightarrow E\rightarrow B\rightarrow 0
$$
in the category $\mathcal{A}^{\alpha,\beta}$ such that $\ch^{\beta}(E)=v$ and
$$
\lambda_{\alpha,\beta,s}(A)=\lambda_{\alpha,\beta,s}(E)=0.
$$
A first observation is that for this choice of invariants, the walls properly intersecting the $(\alpha,s)$--plane are disjoint. Indeed, such walls are given by the equations
$$
\left(s+\frac{1}{6}\right)\alpha^2=\frac{\ch^{\beta}_3(A)}{\ch^{\beta}_1(A)}.
$$
As a consequence we obtain the following special case of Bertram's Lemma:
\begin{lemma}[Bertram's Lemma] Assume that there is a short exact  sequence 
$$
0\rightarrow A \rightarrow E \rightarrow B \rightarrow 0
$$ 
in $\mathcal{A}^{\alpha_0,\beta}$ destabilizing an object $E\in\coh^{\beta}(\mathbb{P}^3)$ with $\ch^{\beta}(E)=(-R,0,D,0)$ with respect to the stability condition $\sigma_{\alpha_0,\beta,s_0}$. If $\ch_1^{\beta}(A)\neq 0$ then $A$ destabilizes $E$ for all $(\alpha,s)\in W_{A,B}=\{(\alpha,s)\colon \lambda_{\alpha,\beta,s}(A)=0\}$.
\end{lemma}
\begin{proof} We only need to prove that $A$, $E$, and $B\in\mathcal{A}^{\alpha,\beta}$ for all $(\alpha,s)\in W_{A,B}$. Since $\ch^{\beta}_1(E)=0$ and $E\in\coh^{\beta}(\mathbb{P}^3)$, then $E\in\mathcal{A}^{\alpha,\beta}$ for all $(\alpha,s)\in W_{A,B}$. Moreover, $E$ must be Bridgeland semistable along $W_{A,B}$ since the walls for $\ch^{\beta}(E)$ are disjoint.

Now, the long exact sequence of $\coh^{\beta}$--cohomologies gives
$$
0\rightarrow \mathcal{H}^{-1}_{\beta}(A)\rightarrow 0 \rightarrow \mathcal{H}^{-1}_{\beta}(B)\rightarrow A \rightarrow E \rightarrow \mathcal{H}^0_{\beta}(B)\rightarrow 0,
$$ 
implying $A\in\coh^{\beta}(\mathbb{P}^3)$. 

Because the category $\mathcal{T}_{\alpha,\beta}$ is closed under quotients then $\mathcal{H}_{\beta}^0(B)\in\mathcal{T}_{\alpha,\beta}$ for all $(\alpha,s)\in W_{A,B}$. Thus to check that $B\in\mathcal{A}^{\alpha,\beta}$ it is enough to check that $\mathcal{H}^{-1}_{\beta}(B)\in\mathcal{F}_{\alpha,\beta}$. If the set
$$
\{(\alpha,s)\in W_{A,B}| \ \mathcal{H}^{-1}_{\beta}(B)\notin \mathcal{F}_{\alpha,\beta}\ \text{for some}\ \alpha>\alpha_0\}
$$
is nonempty, then it has a well defined infimum $(\alpha_1,s_1)$. By definition of infimum, $\mathcal{H}^{-1}_{\beta}(B)$ has a subobject $F_1$ in $\coh^{\beta}(\mathbb{P}^3)$ with $\nu_{\alpha_1,\beta}(F_1)=0$. Thus, the quotient $B/F_1[1]$ in $\coh^{\beta}(\mathbb{P}^3)$ will destabilize $E$ as long as $A\in \mathcal{A}^{\alpha_1,\beta}$. Therefore the set
$$
\{(\alpha,s)\in W_{A,B}| \ A\notin \mathcal{A}^{\alpha,\beta}\ \text{for some}\ \alpha>\alpha_0\}
$$
is nonempty and has a well-defined infimum $(\alpha_2,s_2)$ with $\alpha_2<\alpha_1$. Thus $A$ fits into an exact sequence 
$$
0\rightarrow A'\rightarrow A \rightarrow A''\rightarrow 0
$$
in $\coh^{\beta}(\mathbb{P}^3)$ with $A''$ tilt--semistable of tilt $\nu_{\alpha_2,\beta}(A'')=0$. Therefore,
$$
\lim_{\footnotesize{\begin{matrix}(\alpha,s)\rightarrow (\alpha_2,s_2)\\ \alpha_0<\alpha<\alpha_2\\ (\alpha,s)\in W_{A,B}\end{matrix}}}\lambda_{\alpha,\beta,s}(A'')=-\infty.
$$
As a consequence, for $(\alpha,s)\in W_{A,B}$ close to $(\alpha_2,s_2)$ with $\alpha_0<\alpha<\alpha_2$ we will have $A'$ as a subobject of $E$ in $\mathcal{A}^{\alpha,\beta}$ with
$$
\ch^{\beta}_3(A')-\left(s+\frac{1}{6}\right)\alpha^2\ch^{\beta}_1(A')>0,
$$
contradicting the Bridgeland semistability of $E$ along $W_{A,B}$.

Similarly, one proves that $A,B\in\mathcal{A}^{\alpha,\beta}$ for every $0<\alpha<\alpha_0$.
\end{proof}

\subsection{Limit semistable objects}
The following lemma was proven in \cite{BMS}, it describes the type of objects that remain $\sigma_{\alpha,\beta,s}$--semistable as the parameter $s$ grows.
\begin{lemma}[{\cite[Lemma 8.9]{BMS}}]\label{limit_semistable_objects} If $E\in\mathcal{A}^{\alpha,\beta}$ is $\sigma_{\alpha,\beta,s}$--semistable for all $s\gg 0$, then one of the following conditions holds.
\begin{enumerate}
\item $E=\mathcal{H}^0_{\beta}(E)$ is $\nu_{\alpha,\beta}$--semistable object.
\item $\mathcal{H}^{-1}_{\beta}(E)$ is $\nu_{\alpha,\beta}$--semistable and $\mathcal{H}^0_{\beta}(E)$ is either $0$ or a sheaf supported in dimension 0.
\end{enumerate}
\end{lemma} 
\begin{proposition}
If $\ch^{\beta}(E)=(-R,0,D,0)$ for some $R\geq 0$ and $D>0$ and $E$ is $\sigma_{\alpha,\beta,s}$--semistable for all $s\gg 0$ then $E\in\coh^{\beta}(\mathbb{P}^3)$. Moreover, if $R=0$ then $E$ is a Gieseker semistable sheaf.
\end{proposition}
\begin{proof}
First, notice that if $E$ is not in $\coh^{\beta}(\mathbb{P}^3)$ then, by Lemma \ref{limit_semistable_objects}, $E$ fits into an exact sequence 
$$
0\rightarrow F[1] \rightarrow E \rightarrow T \rightarrow 0
$$
in $\mathcal{A}^{\alpha,\beta}$ with $F\in\mathcal{F}_{\alpha,\beta}$ a $\nu_{\alpha,\beta}$-semistable object and $T\in \mathcal{T}_{\alpha,\beta}$ a 0-dimensional sheaf. But this implies that $\ch_1^{\beta}(F)=0$ contradicting that $\nu_{\alpha,\beta}(F)\leq 0$. Thus $E\in \coh^{\beta}(\mathbb{P}^3)$ and $E$ is $\nu_{\alpha,\beta}$-semistable. 

Now, assume that $R=0$. Notice that in this case $E$ must be a sheaf since these are the only objects in $\coh^{\beta}(\mathbb{P}^3)$ with $\ch^{\beta}_0=\ch^{\beta}_1=0 $.

Since every short exact sequence of sheaves
$$
0\rightarrow A\rightarrow E\rightarrow B\rightarrow0
$$
stays exact in $\mathcal{A}^{\alpha,\beta}$ then $E$ must be a pure 1-dimensional sheaf. Indeed, a 0-dimensional torsion subsheaf of $E$ would $\sigma_{\alpha,\beta,s}$-destabilize $E$ for all $s>0$. Thus, $\ch^{\beta}_2(A)\neq 0$ and 
$$
\lambda_{\alpha,\beta,s}(A)\leq 0
$$
due to the semistability of $E$. The condition on the Bridgeland slope of $A$ is equivalent to
$$
\frac{\ch_3(A)}{\ch_2(A)}\leq \beta=\frac{\ch_3(E)}{\ch_2(E)},
$$
and so $E$ is Gieseker semistable.
\end{proof}

\subsection{Destabilizing objects with $\ch^{\beta}_1=0$}
As mentioned before, if $A$ is a destabilizing subobject in $\mathcal{A}^{\alpha,\beta}$ of an object $E\in\coh^{\beta}(\mathbb{P}^3)$ of invariants $\ch^{\beta}(E)=(-R,0,D,0)$, then $A\in\coh^{\beta}(\mathbb{P}^3)$ and therefore $\ch^{\beta}_1(A)\geq 0$. In the case when $\ch^{\beta}_1(A)=0$ then $A$ would make $E$ properly semistable for every $\alpha,s>0$ for which $E$ is semistable. 

Assume that $E$ is $\lambda_{\alpha,\beta,s}$-semistable for $s\gg 0$ (or equivalently for $\alpha\gg 0$). Notice that if $\ch^{\beta}_1(A)=0$ and $\lambda_{\alpha,\beta,s}(A)=0$ then $\ch^{\beta}_3(A)=0$. Moreover, $A$ is semistable for $s\gg0$ since otherwise every destabilizing subobject of $A$ would also destabilize $E$. Thus $-rd\geq 0$ because such $A$ would be $\nu_{\alpha,\beta}$-semistable. 

On the other hand, since $A$ is a subobject of $E$ in $\mathcal{A}^{\alpha,\beta}$ then
$$
0<d-\frac{\alpha^2}{2}r\leq D+\frac{\alpha^2}{2}R.
$$
Taking limits as $\alpha\rightarrow 0$ and $\alpha\rightarrow \infty$ we obtain
\begin{eqnarray}
& &0\leq d\leq D,\\
& & 0\leq -r\leq R.
\end{eqnarray}
In the particular case that $R=0$ we would get $r=0$ and so $A$ would be an object with Chern character $\ch^{\beta}(A)=(0,0,d,0)$ that is semistable for all $s\gg 0$, i.e., a Gieseker semistable sheaf destabilizing $E$ with respect to Gieseker semistability.


\subsection{Bounding destabilizing objects with $\ch^{\beta}_1> 0$}

Recall that if $E$ is a limit semistable object of Chern character $(-R,0,D,0)$ then the vertical $\alpha$-walls are finite when $s>1/3$ and so the walls destabilizing $E$ in the complete $(\alpha,s)$-slice are also finite. This implies the finiteness of the vertical $\alpha$-walls for $E$ for every $s>0$.

In what follows we will run an algorithm in order to restrict the Chern characters of possible destabilizing subobjects of a limit semistable object $E$ and prove the finiteness of walls for the Chern character $(-R,0,D,0)$ in the $(\alpha,s)$-slice.

Suppose that we have a destabilizing sequence in $\mathcal{A}^{\alpha,\beta}$
$$
0\rightarrow A\rightarrow E\rightarrow B\rightarrow 0
$$
where $\ch(A)=(r,c,d,e)$ with $c>0$ and $\ch(E)=(-R,0,D,0)$ with $R\geq 0,\ D>0$, and assume both $A,E\in\coh^{\beta}(\mathbb{P}^3)$. Under these assumptions we have
\begin{eqnarray}
0<d-\frac{\alpha^2}{2}r<D+\frac{\alpha^2}{2}R \label{category_condition}\\
Q_{\alpha,\beta,K}(A)\geq 0 \label{Q_for_A}\\
Q_{\alpha,\beta,K}(B)\geq 0 \label{Q_for_B}\\
\left(s+\frac{1}{6}\right)\alpha^2=\frac{e}{c}.\label{wall_condition}
\end{eqnarray}

By Bertram's lemma we know that $A$ destabilizes $E$ for each $(\alpha,s)\in W_{A,B}$. Then taking the limit as $\alpha\rightarrow 0$ we obtain
\begin{equation}\label{bounds_for_d}
0\leq d\leq D
\end{equation}
Though the value of $\beta$ may seem arbitrary, it is a very specific one depending on the invariants of $E$. Indeed,
$$
\beta=\begin{cases}
\frac{\ch_1(E)}{R}& \text{if}\ R>0\\
\frac{\ch_3(E)}{D}& \text{if}\ R=0.
\end{cases}
$$
In particular, $r,c,d$ and $e$ are rational numbers with bounded denominators and so  inequality \eqref{bounds_for_d} tells us that there are only finitely many possibilities for $d$.

Now, from Lemma \ref{qform} we know that $0\leq Q_{\alpha,\beta,K}(A)\leq Q_{\alpha,\beta,K}(E)$, which translates into
\begin{equation}\label{Q_A_and_Q_E}
0\leq K\alpha^2(c^2-2rd)+4d^2-6ce\leq 2K\alpha^2 RD+4D^2.
\end{equation}
Fixing $K$, and using Bertram's lemma again we take the limit as $\alpha\rightarrow 0$ to obtain
\begin{equation}\label{bounds_for_ce}
0\leq 4d^2-6ce\leq 4D^2.
\end{equation}
On the other hand, the wall condition \eqref{wall_condition} implies that $e$ and $c$ have the same sign, and moreover, that $e,c>0$ since $A\in\coh^{\beta}(\mathbb{P}^3)$. Combining this with  \eqref{bounds_for_ce} we obtain
\begin{equation}\label{estimate_ce_1}
0<c(6e)\leq 4d^2.
\end{equation}
Therefore there are only finitely many possibilities for $c$ and $e$.

To bound $r$, we take the limit as $s\rightarrow 0$ in \eqref{wall_condition}, which is possible by Bertram's lemma, and replace the corresponding value of $\alpha^2$ in \eqref{category_condition} to obtain
\begin{equation}\label{inequality-rank}
-\frac{c(2D-2d)}{6e}-R\leq r\leq \frac{c(2d)}{6e}.
\end{equation}
Thus there are only finitely many possibilities for $\ch^{\beta}(A)$ and so only finitely many walls destabilizing limit semistable objects of Chern character $(-R,0,D,0)$ in the $(\alpha,s)$--plane.

These estimates can actually be improved further. Indeed, plugging in the wall equation \eqref{wall_condition} into \eqref{Q_for_B} and taking $K=1$ we obtain
\begin{equation}\label{Q(B)}
(d-D)\left(2\alpha^2(R+r)+4(D-d)\right)\leq -6ec\left(\frac{6s}{6s+1}\right)<0.
\end{equation}
This inequality combined with inequalities \eqref{Q_A_and_Q_E} and \eqref{bounds_for_d} give
\begin{equation}\label{sharp-d}
0<2d<2D.
\end{equation}
On the other hand, taking $\alpha\rightarrow 0$ and $s\rightarrow \infty$ in \eqref{Q(B)} we obtain
$$
0< c(6e)\leq 4(D-d)^2,
$$
and so 
\begin{equation}\label{estimate-min}
0<c(6e)\leq \min\left\{4d^2,4(D-d)^2\right\}.
\end{equation}

\section{The case $X=\mathbb{P}^3$ and $\beta=0$}\label{sec:beta=0}
For the remaining of the paper we will focus on the case $X=\mathbb{P}^3$ and $\beta=0$. Notice that not every numerical solution to inequalities \eqref{inequality-rank},  \eqref{sharp-d}, and  \eqref{estimate-min} corresponds to an actual wall. For instance, the vector $w=(r,c,d,e)$ should be the Chern character of an object in $\mathcal{A}^{\alpha,0}$ and therefore:
\begin{eqnarray}\label{chern-classes-conditions}
d-\frac{c^2}{2}\in \mathbb{Z} &\text{because}\ c_2(w)\ \text{is an integer},\label{c2-condition}\\
2e-cd+\frac{c^3}{6}\in \mathbb{Z} &\text{because}\ c_3(w)\ \text{is an integer},\label{c3-condition}\\
e-\frac{c}{6}\in \mathbb{Z} &\text{because}\ \chi(w)\ \text{is an integer}.\label{chi-condition}
\end{eqnarray}

\begin{lemma}\label{first-vanishing}
Suppose that $F$ is a sheaf with $\ch(F)=(R,0,-D,0)$ such that $F[1]$ is $\lambda_{\alpha,0,s}$-semistable for all $s\gg 0$ (or equivalently for all $\alpha\gg 0$). Then
$$
H^0(F(-1))=0.
$$
\end{lemma}

\begin{proof}
Suppose that $H^0(F(-1))\neq 0$, then there is a nonzero morphism 
$$
\mathcal{O}(1)[1]\rightarrow F[1]
$$ 
in $\mathcal{A}^{\alpha,0}$ for all $\alpha\gg 0$. Since $\mathcal{O}(1)[1]$ is $\lambda_{\alpha,0,s}$-stable,  $F[1]$ is $\lambda_{\alpha,0,s}$-semistable and 
$$
\lambda_{\alpha,0,s}(\mathcal{O}(1)[1])=\frac{(6s+1)\alpha^2-1}{3(\alpha^2-1)}>0\ \ \text{for all}\ \alpha\gg 0,
$$ 
then such morphism cannot exist.
\end{proof}

Further vanishing can be obtained by using the estimates for walls. Indeed, we have the following:

\begin{lemma}\label{second-vanishing} Suppose that $E$ is an asymptotically $\lambda_{\alpha,0,s}$-semistable object with $\ch(E)=(-R,0,D,0)$, where $R>0$ and $D>1$. Then
\begin{enumerate}
\item $\Hom(\mathcal{O}(n),E)=0$ for $n\geq D+1$, and
\item $ \Hom(\mathcal{O}(n),E)=0$  for $n\geq \frac{D}{2}+1$ when $D$ is even.
\end{enumerate}
\end{lemma}

\begin{proof}
First, notice that if $n>0$ then $\mathcal{O}(n)\in\mathcal{A}^{\alpha,0}$ is $\lambda_{\alpha,0,s}$-stable for all $\alpha<n$. On the other hand, 
$$
\lambda_{\alpha,0,s}(\mathcal{O}(n))> 0\ \ \text{if and only if}\ \ \left(s+\frac{1}{6}\right)\alpha^2< \frac{n^2}{6}.
$$
Then it follows that if $(r,c,d,e)$ is the Chern character of a destabilizing subobject producing the largest wall for $E$ and
$$
\frac{e}{c}=\frac{\alpha_0^2}{6},
$$
then $\Hom(\mathcal{O}(n),E)=0$ for all $n>\alpha_0$. 

To start, let us prove that $\alpha_0\leq D$ and so a nonzero morphism $\mathcal{O}(n)\rightarrow E$ for $n> D$ can not exist. This easily follows from inequalities \eqref{sharp-d} and \eqref{estimate-min} since the solution producing the largest possible numerical value of $e/c$ is $2d=D$ , $c=1$, and $6e=D^2$.

Now, assume that $D$ is even. We will prove that in this case $\alpha_0\leq D/2$. Indeed, the numerical solution $(r,1,D/2,D^2/6)$ is not a Chern character because it does not satisfy condition \eqref{c2-condition}. The next largest possible wall occurs when either $2d=D$ and $c$ is even or when $2d=D-1$ and $c$ is odd. In the first case, the largest possible value for $e/c$ is obtained when $c=2$ and $6e=D^2/2$, in which case
$$
\frac{e}{c}=\frac{1}{6}\left(\frac{D}{2}\right)^2.
$$
For the second case, $e/c$ would be largest when $c=3$, in which case $2e$ would have to be an integer for condition \eqref{c3-condition} to be satisfied, and
$$
e=\frac{1}{2}\floor*{\frac{(D-1)^2}{9}}.
$$ 
However, since 
$$
\frac{1}{6}\floor*{\frac{(D-1)^2}{9}}<\frac{1}{6}\left(\frac{D}{2}\right)^2
$$
then we conclude that $E$ is $\lambda_{\alpha,0,s}$-semistable for all
$$
\left(s+\frac{1}{6}\right)\alpha^2\geq \frac{1}{6}\left(\frac{D}{2}\right)^2.
$$
Thus $\alpha_0\leq D/2$ when $D$ is even.
\end{proof}

\begin{corollary}
An object $E[1]$ with $\ch(E[1])=(-2,0,2,0)$ is asymptotically $\lambda_{\alpha,0,s}$-semistable if and only if $E$ is an instanton sheaf of charge 2.
\end{corollary}

\begin{remark} We know by \cite[Example 8.1.2]{Hart80} that for every $m>0$, there exists a semistable vector bundle $F$ with $\ch(F)=(2,0,-2m,-m(m-1)(m+4))$ and $H^1(F(-\frac{D}{2}-1))\neq 0$. Thus for $D=2$, the object $A=F[1]$ is not limit semistable. Moreover, $A$ satisfies all the conditions of Theorem 3.1 but condition (3), and so this condition can not be relaxed to allow $\mathcal{H}^{-1}(A)$ to be $\mu$-semistable.
\end{remark}

\begin{remark} In \cite[Example 8.1.3]{Hart80}, Hartshorne also constructs for every $D$ even, a stable rank 2 vector bundle $F$ with $\ch_1(F)=0$ and $H^1(F(-\frac{D}{2}))\neq 0$. Thus $A=F[1]$ is strictly semistable along the numerical wall 
$$
\left(s+\frac{1}{6}\right)\alpha^2=\frac{1}{6}\left(\frac{D}{2}\right)^2.
$$
\end{remark}

\begin{example} If $v=(R,0,D,0)$ with $R,D>0$, then there are no asymptotically semistable objects along $\beta=0$ since such object would necessarily have to be tilt semistable and so would satisfy the Bogomolov inequality. However, for small values of $\alpha$ there are Bridgeland stable objects. For instance, if $d\geq 2$ then those complexes $E$ fitting into a short exact sequence of the form
$$
0\rightarrow \mathcal{O}(-d)[2]\rightarrow E\rightarrow \mathcal{O}(d)\rightarrow 0
$$
are Bridgeland semistable for $\beta=0$ and $0<d^2-(6s+1)\alpha^2\ll 1$.
\end{example}

\begin{example}\label{walls-charge1} If $\beta=0$ and $v=(-2,0,1,0)$, then there is only one wall with $c>0$ produced by an object of invariants 
$$
\ch(A)=\left(r,1,\frac{1}{2},\frac{1}{6}\right),\ \ -3\leq r\leq 1.
$$
\end{example}

Since a properly torsion free instanton sheaf $E$ of charge 1 fits into an exact sequence 
$$
0\rightarrow \mathcal{O}_L(1)\rightarrow E[1]\rightarrow 2\mathcal{O}[1]\rightarrow 0
$$
in $\mathcal{A}^{\alpha,0}$, then $E[1]$ and $E^{\vee}[1]$ are both strictly $\lambda_{\alpha,0,s}$-semistable and $S$-equivalent. Thus, Example \ref{walls-charge1} together with Proposition \ref{LSSO_(-2,0,1,0)} imply the following.


\begin{corollary}\label{chamber0}
For $\beta=0$, the $(\alpha,s)$ plane is divided into two stability chambers for the numerical Chen character $v:=(-2,0,1,0)$:
\begin{itemize}
\item[(C1)] $\calm_{\alpha,0,s}(v)=\p5$, and $\labs$-stable objects are of the form 
$E[1]$ where $E$ is either a locally free instanton sheaf of charge 1, or the dual of a properly torsion free sheaf of charge 1;
\item[(C2)] $\calm_{\alpha,0,s}(v)=\emptyset$.
\end{itemize}
\end{corollary}


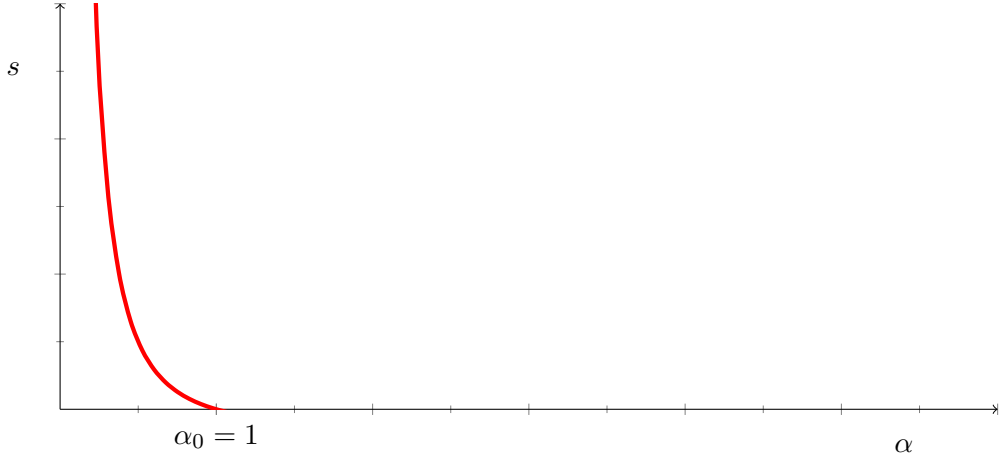
\begin{figure}[ht]\label{fig (2,0,-1,0)}
\begin{tikzpicture}
 
\begin{axis}[
    xmin = 0, xmax = 6.0,
    ymin = 0, ymax = 3.0,
    xtick distance = 1.0,
    ytick distance = 1.0,
    grid = none,
    minor tick num = 1,
    width = \textwidth,
    height = 0.5\textwidth,
    axis lines=middle,
    axis line style={->},
    x label style={at={(axis description cs:0.9,-0.05)},anchor=north},
    y label style={at={(axis description cs:-0.05,.8)},anchor=south},
    xlabel={$\alpha$},
    ylabel={$s$},
    xticklabels={,,},
    yticklabels={,,},
    extra x ticks={1},
extra x tick style={xticklabel={$\alpha_0=1$}}]
    \addplot[
        domain = 0:30,
        samples = 400,
        smooth,
        ultra thick,
        red,
    ] {((1/x^2)-1)/6};
    \end{axis}
\end{tikzpicture}
  \caption{The only $\lambda$-wall for $(-2,0,1,0)$ is the monad wall.}
\end{figure}

The two stability chambers described above are pictured in Figure \ref{fig (2,0,-1,0)} as the components of the complement of the red curve, which corresponds to the \emph{monad wall} described in Remark \ref{monad-wall}. The chamber (C1) is the outermost chamber ($\alpha\gg 0$) and the chamber (C2) is the one enclosed by the axes and the monad wall, here the moduli space is empty by Lemma \ref{quiver_region}.

\begin{example}\label{walls-charge2} If $\beta=0$ and $v=(-2,0,2,0)$, then the possible Chern characters of subobjects with $c>0$ and destabilizing a limit semistable object of Chern character $v$ are given by inequalities \eqref{inequality-rank}, \eqref{sharp-d} and \eqref{estimate-min}, and conditions \eqref{c2-condition}, \eqref{c3-condition} and \eqref{chi-condition}:
\\
\small{
\begin{center}
\begin{tabular}{|c|c|c|c|}
\hline
& & &
\\
$\ch(A)$ & $\ch_0(A)$ & $\alpha_0$ & Possible destabilizing sequence\\
 & & & 
\\
\hline
& & & 
\\
$\displaystyle \left(r,1,\frac{1}{2},\frac{1}{6}\right)$ & $\displaystyle -5\leq r\leq 1$ & 1 & $\displaystyle 0\rightarrow  \mathcal{O}(1)\rightarrow E\rightarrow \mathcal{G}\rightarrow 0$\\
& & & \\
$\displaystyle \left(r,2,1,\frac{1}{3}\right)$ & $-4\leq r\leq 2$ & & $0\rightarrow \mathcal{K}[1]\rightarrow E\rightarrow 2\mathcal{O}(-1)[2]\rightarrow 0$, \ $\mathcal{K}[1]\in \Ext^1(6\mathcal{O}[1],2\mathcal{O}(1))$  \\
& & &\\
$\displaystyle \left(r,1,\frac{3}{2},\frac{1}{6}\right)$ & $-3\leq r\leq 3$ & & $0\rightarrow \mathcal{G}^{\vee}[2]\rightarrow E^{\vee}[2]\rightarrow \mathcal{O}(-1)[2]\rightarrow 0$ \\
& & &\\
\hline
\end{tabular}
\end{center}
}
\normalsize
\end{example}

\begin{lemma}\label{quiver_region}
There are no walls for the Chern character $v=(-R,0,D,0)$ in the region given by
$$
\left(s+\frac{1}{6}\right)\alpha^2<\frac{1}{6},
$$
i.e., the monad wall $(\alpha_0=1)$ is the last wall. Moreover, there are no semistable objects of Chern character $v$ in this region.
\end{lemma}
\begin{proof}
First, notice that $\mathcal{O}(1), \mathcal{O}[1], \mathcal{O}(-1)[2]$, and $\mathcal{O}(-2)[2]$ are $\lambda_{\alpha,0,s}$-stable objects in $\mathcal{A}^{\alpha,0}$ for all $\alpha<1$, in particular, for all $(\alpha,s)$ in the region 
$$
Q=\left\{(\alpha,s)\ \bigg|\ \alpha,s>0,\ \left(6s+1\right)\alpha^2\leq 1\right\}.
$$
Now, recall that 
$$
\lambda_{\alpha,0,s}(\mathcal{O}(n)[k])=\frac{n^3-\left(6s+1\right)\alpha^2n}{3(n^2-\alpha^2)},
$$
thus a simple computation shows that for all $(\alpha,s)\in Q$ we have
$$
\lambda_{\alpha,0,s}(\mathcal{O}(-2)[2])<\lambda_{\alpha,0,s}(\mathcal{O}(-1)[2])\leq 0=\lambda_{\alpha,0,s}(\mathcal{O}[1])\leq \lambda_{\alpha,0,s}(\mathcal{O}(1)).
$$
Therefore, if $\lambda_0$ is any real number between $\lambda_{\alpha,0,s}(\mathcal{O}(-2)[2])$ and $\lambda_{\alpha,0,s}(\mathcal{O}(-1)[2])$ and $\mathcal{A}^{\#}_{\alpha,0,s}$ denotes the tilt of $\mathcal{A}^{\alpha,0}$ with respect to the Bridgeland slope at $\lambda_0$, then
$$
\langle \mathcal{O}(-2)[3], \mathcal{O}(-1)[2], \mathcal{O}[1], \mathcal{O}(1)\rangle=\mathcal{A}^{\#}_{\alpha,0,s}.
$$
Now, if $E$ is $\lambda_{\alpha,0,s}$-semistable with $\lambda_{\alpha,0,s}(E)=0$ then $E$ and every possible destabilizing subobject of $E$ belong to $\mathcal{A}^{\#}_{\alpha,0,s}$. Thus, $E$ (and every possible destabilizing subobject of $E$) is quasi-isomorphic to a complex of the form
$$
\mathcal{O}(-2)^{\oplus n_{-3}}\rightarrow \mathcal{O}(-1)^{\oplus n_{-2}}\rightarrow \mathcal{O}^{\oplus n_{-1}}\rightarrow \mathcal{O}(1)^{\oplus n_0}.
$$
Denote by $\hat{n}(E)$ the dimension vector $(n_{-3},n_{-2},n_{-1},n_0)$. Then the $\lambda_{\alpha,0,s}$-semistability of $E$ is equivalent to having $\hat{n}(A)\cdot \theta_{\alpha,s}\leq 0$ for every subcomplex $A$ of E, where
$$
\theta_{\alpha,s}=\left(8-2\left(6s+1\right)\alpha^2, -1+\left(6s+1\right)\alpha^2,0,1-\left(6s+1\right)\alpha^2\right).
$$
In the special case when $\ch(E)=(-R,0,D,0)$, an easy computation shows that
$$
\hat{n}(E)=(0,D,2D+R,D).
$$
Thus, every $\lambda$-wall is produced by a subcomplex $A$ with $n_{-3}(A)=0$ and moreover
\begin{align*}
\ch_0(A)&= n_{-2}(A)-n_{-1}(A)+n_0(A)\\
\ch_1(A)&= -n_{-2}(A)+n_0(A)\\
\ch_2(A)&= \frac{1}{2}n_{-2}(A)+\frac{1}{2}n_0(A)\\
\ch_3(A)&= -\frac{1}{6}n_{-2}(A)+\frac{1}{6}n_0(A).
\end{align*}
In particular, $\ch_1(A)=6\ch_3(A)$ and so the equation of the wall produced by $A$ is
$$
\left(s+\frac{1}{6}\right)\alpha^2=\frac{1}{6}.
$$
Now, a complex $E$ of the form
$$
\mathcal{O}(-1)^{\oplus D}\rightarrow \mathcal{O}^{2D+R}\rightarrow \mathcal{O}(1)^{D}
$$
has necessarily $\mathcal{O}(1)$ as a subcomplex. Therefore, such object can not be $\lambda_{\alpha,0,s}$-semistable for $(\alpha,s)$ in the chamber contained in $Q$, since
$$
(0,0,0,1)\cdot\theta_{\alpha,s}=1-(6s+1)\alpha^2>0.
$$
\end{proof}
\begin{remark} If $\mathcal{F}$ is a Gieseker semistable sheaf with $\ch(\mathcal{F})=(0,0,D,0)$ that is not destabilized for any $\alpha>1$, then it follows from the proof of Lemma \ref{second-vanishing} that $\mathcal{F}$ is a rank 0 instanton sheaf. Therefore, all the non-instanton sheaves must be destabilized before the monad wall. 
\end{remark}
\begin{proposition} Suppose that $\nu_{\alpha,0}(E)\neq 0$. Then 
$$
E\ \text{is}\ \lambda_{\alpha,0,s}\text{-semistable}\ \Longleftrightarrow\ E^{\vee}[2]\ \text{is}\ \lambda_{\alpha,0,s}\text{-semistable}.
$$
\end{proposition}
\begin{proof} From \cite[Proposition 4.3.6 and Remark 4.4.3]{BMT} we know that $\sigma_{\alpha,0,s}$ is a self-dual stability condition, thus
$$
\mathbb{D}(\mathcal{P}(0,1))[1]=\mathcal{P}(-1,0)[1]=\mathcal{P}(0,1),
$$
where $\mathbb{D}(E)=R\mathcal{H}om(E,\mathcal{O})[1]$. This says exactly that if $E$ is a $\lambda_{\alpha,0,s}$-semistable object with phase in $(0,1)$ then so is $E^{\vee}[2]$. 
\end{proof}
\begin{corollary} If $0\rightarrow A\rightarrow E\rightarrow B\rightarrow 0$ is a destabilizing short exact sequence in $\mathcal{A}^{\alpha,0}$ for an object $E$ with $\ch(E)=(-R,0,D,0)$ then the dual short exact sequence $0\rightarrow B^{\vee}[2]\rightarrow E^{\vee}[2]\rightarrow A^{\vee}[2]\rightarrow 0$ is a destabilizing sequence for $E^{\vee}[2]$ producing the same wall.
\end{corollary}

\begin{example} Walls for $v=(0,0,3,0)$. In this case there are exactly two actual walls produced by the Chern characters in the table below. The first wall ($\alpha_0=\sqrt{7}$) contracts the locus of structure sheaves of planar cubics to ${\mathbb{P}^3}^*$. This is the only non-instanton component of the Gieseker moduli \cite[Proposition 19]{JMT}. Crossing this wall produces a nontrivial extension $E$ fitting into a diagram of the form
$$
 \xymatrix@=1.5pc{&\mathcal{I}_{p,H}(2)\ar[d]\ar@{=}[r]&\mathcal{I}_{p,H}(2)\ar[d]\\
\mathcal{O}_{H}(-1)[1]\ar[r]\ar@{=}[d]& E\ar[d]\ar[r]& \mathcal{O}_H(2)\ar[d]\\
\mathcal{O}_{H}(-1)[1]\ar[r]&U\ar[r]& \mathbb{C}_p }
$$
Notice that $\mathcal{O}_H(-1)[1]$, $\mathcal{O}_H(2)$, and $\mathcal{I}_{p,H}(2)$ are all unstable below the monad wall. From the exact sequences
\begin{align*}
    &0\rightarrow 2\mathcal{O}_H(1)\rightarrow\mathcal{I}_{p,H}(2)\rightarrow \mathcal{O}_H[1]\rightarrow 0,\\
    &0\rightarrow \mathcal{O}(-1)[1]\rightarrow \mathcal{O}_H(-1)[1]\rightarrow \mathcal{O}(-2)[2] \rightarrow 0,\\
    &0\rightarrow \mathcal{O}(1)\rightarrow \mathcal{O}_H(1)\rightarrow \mathcal{O}[1] \rightarrow 0,\\
    &0\rightarrow \mathcal{O}[1]\rightarrow \mathcal{O}_H[1]\rightarrow \mathcal{O}(-1)[2] \rightarrow 0,
\end{align*}
one sees that the stable factors of all such $E$ at the monad wall are precisely the generators of the quiver heart described in Lemma  \ref{quiver_region}.
\begin{center}
\begin{tabular}{|c|c|c|c|}
\hline
& & &
\\
$\ch(A)$ & $\ch_0(A)$ & $\alpha_0$ & Possible destabilizing sequence\\
 & & & 
\\
\hline
& & & 
\\
$\displaystyle \left(r,1,\frac{3}{2},\frac{7}{6}\right)$ & $\displaystyle  r=0$ & $\sqrt{7}$ & $\displaystyle 0\rightarrow  \mathcal{O}_H(2)\rightarrow \mathcal{O}_C(2)\rightarrow \mathcal{O}_H(-1)[1]\rightarrow 0$,\ $\deg(C)=3$.\\
& & & \\
\hline
& & & \\
$\displaystyle \left(r,1,\frac{1}{2},\frac{1}{6}\right)$ & $\displaystyle -5\leq r\leq 1$ & 1 & $ 0\rightarrow  \mathcal{O}(1)\rightarrow E\rightarrow \mathcal{R}[1]\rightarrow 0$\\
& & & \\
$\displaystyle \left(r,2,1,\frac{1}{3}\right)$ & $-4\leq r\leq 2$ & & $ 0\rightarrow  2\mathcal{O}(1)\rightarrow E\rightarrow \mathcal{S}[1]\rightarrow 0$\\
& & &\\
$\displaystyle \left(r,3,\frac{3}{2},\frac{1}{2}\right)$ & $-3\leq r\leq 3$ & & $ 0\rightarrow  3\mathcal{O}(1)\rightarrow E\rightarrow \mathcal{K}[1]\rightarrow 0$\\
& & &\\
$\displaystyle \left(r,1,\frac{5}{2},\frac{1}{6}\right)$ & $-1\leq r\leq 5$ & & $ 0\rightarrow  \mathcal{R}^{\vee}[1]\rightarrow E^{\vee}[2]\rightarrow \mathcal{O}(-1)[2]\rightarrow 0$\\
& & &\\
$\displaystyle \left(r,2,2,\frac{1}{3}\right)$ & $-2\leq r\leq 4$ & & $ 0\rightarrow  \mathcal{S}^{\vee}[1]\rightarrow E^{\vee}[2]\rightarrow 2\mathcal{O}(-1)[2]\rightarrow 0$\\
& & &\\
$\displaystyle \left(r,1,\frac{3}{2},\frac{1}{6}\right)$& $-3\leq r\leq 3$ &1 & $\displaystyle 0\rightarrow \mathcal{I}_{p,H}(2)\rightarrow E\rightarrow U\rightarrow 0$\\
& & &\\
\hline
\end{tabular}
\end{center}

\end{example}

\bibliographystyle{hep}
\bibliography{referencesJMM21.bib}

\begin{thebibliography}{{Pre}21}

\bibitem[BM14]{BM14a}
A.~Bayer and E.~Macr{\`\i}, \textsl{ MMP for moduli of sheaves on K3s via
  wall-crossing: nef and movable cones, Lagrangian fibrations},
\newblock Inventiones mathematicae \textbf{ 198}(3), 505--590 (2014).

\bibitem[BMS16]{BMS}
A.~Bayer, E.~Macr{\`\i} and P.~Stellari, \textsl{ The space of stability
  conditions on abelian threefolds, and on some Calabi-Yau threefolds},
\newblock Inventiones mathematicae \textbf{ 206}(3), 869--933 (2016).

\bibitem[BMT14]{BMT}
A.~Bayer, E.~Macr{\`{\i}} and Y.~Toda, \textsl{ Bridgeland stability conditions
  on threefolds {I}: {B}ogomolov-{G}ieseker type inequalities},
\newblock J. Algebraic Geom. \textbf{ 23}(1), 117--163 (2014).

\bibitem[Bri07]{B07}
T.~Bridgeland, \textsl{ Stability conditions on triangulated categories},
\newblock Ann. of Math. (2) \textbf{ 166}(2), 317--345 (2007).

\bibitem[Bri08]{B08}
T.~Bridgeland, \textsl{ Stability conditions on {$K3$} surfaces},
\newblock Duke Math. J. \textbf{ 141}(2), 241--291 (2008).

\bibitem[Cha84]{Chang}
M.~Chang, \textsl{ Stable rank 2 reflexive sheaves on $\mathbb{P}^3$ with small
  $c_2$ and applications},
\newblock Transactions of the American Mathematical Society \textbf{ 284},
  57--89 (1984).

\bibitem[GJ16]{GJ}
M.~Gargate and M.~Jardim, \textsl{ Singular loci of instanton sheaves on
  projective space},
\newblock International Journal of Mathematics \textbf{ 27}(07), 1640006
  (2016), {https://doi.org/10.1142/S0129167X16400061}.

\bibitem[Har80]{Hart80}
R.~Hartshorne, \textsl{ Stable reflexive sheaves},
\newblock Mathematische Annalen \textbf{ 254}(2), 121--176 (1980).

\bibitem[Jar04]{J-i}
M.~Jardim, \textsl{ Instanton sheaves on complex projective spaces},
\newblock Collectanea Mathematica \textbf{ 57}, 69--91 (2004).

\bibitem[JM19]{JM19}
M.~{Jardim} and A.~{Maciocia}, \textsl{ {Walls and asymptotics for Bridgeland
  stability conditions on 3-folds}},
\newblock arXiv e-prints , arXiv:1907.12578 (July 2019), {1907.12578}.

\bibitem[JMT17]{JMT}
M.~Jardim, M.~Maican and A.~S. Tikhomirov, \textsl{ {Moduli spaces of rank 2
  instanton sheaves on the projective space}},
\newblock Pacific J. Math. \textbf{ 291}, 399--424 (2017), {1702.06553}.

\bibitem[Li19a]{Li2}
C.~Li, \textsl{ On stability conditions for the quintic threefold},
\newblock Inventiones mathematicae \textbf{ 218}(1), 301--340 (2019).

\bibitem[Li19b]{Li}
C.~Li, \textsl{ Stability conditions on fano threefolds of picard number 1},
\newblock JEMS. \textbf{ 21}(3), 709--726 (2019).

\bibitem[Mac14]{M-p3}
E.~Macr{\`{\i}}, \textsl{ A generalized {B}ogomolov-{G}ieseker inequality for
  the three-dimensional projective space},
\newblock Algebra Number Theory \textbf{ 8}(1), 173--190 (2014).

\bibitem[MP16]{MP}
A.~Maciocia and D.~Piyaratne, \textsl{ Fourier-{M}ukai transforms and
  Bridgeland stability conditions on abelian threefolds II},
\newblock International Journal of Mathematics \textbf{ 27}(01), 1650007
  (2016), {https://doi.org/10.1142/S0129167X16500075}.

\bibitem[{Pre}21]{P}
V.~{Pretti}, \textsl{ {Zero Rank Asymptotic Bridgeland Stability}},
\newblock arXiv e-prints , arXiv:2104.08946 (April 2021), {2104.08946}.

\bibitem[Sch14]{S-q3}
B.~Schmidt, \textsl{ {A generalized Bogomolov-{G}ieseker inequality for the
  smooth quadric threefold}},
\newblock Bulletin of the London Mathematical Society \textbf{ 46}(5), 915--923
  (06 2014),
  {https://academic.oup.com/blms/article-pdf/46/5/915/6694341/bdu048.pdf}.

\bibitem[Sch20]{S}
B.~Schmidt, \textsl{ Bridgeland stability on threefolds: Some wall crossings},
\newblock J. Algebraic Geom. \textbf{ 29}(2), 247--283 (2020).

\end{thebibliography}

\end{document}